\documentclass[11pt,reqno]{amsart}
\usepackage[utf8]{inputenc}
\usepackage{amsopn,amssymb,mathrsfs,xcolor,url,enumitem,accents}
\usepackage[abbrev]{amsrefs} 
\usepackage[a4paper,lmargin=3cm,rmargin=3cm,tmargin=4cm,bmargin=4cm,marginparwidth=2.8cm,marginparsep=1mm]{geometry} 
\usepackage[bookmarks=true,hyperindex,pdftex,colorlinks,citecolor=blue]{hyperref} 
\usepackage{cleveref}

\makeatletter
\@namedef{subjclassname@2020}{\textup{2020} Mathematics Subject Classification}
\makeatother

\usepackage{marginnote} 

\newtheorem{theorem}{Theorem}[section]
\newtheorem{corollary}[theorem]{Corollary}
\newtheorem{lemma}[theorem]{Lemma}
\newtheorem{proposition}[theorem]{Proposition}
\newtheorem{problem}[theorem]{Problem}

\theoremstyle{definition}
\newtheorem{definition}[theorem]{Definition}
\newtheorem{example}[theorem]{Example}

\newcommand{\cl}[1]{\ensuremath{\overline{{#1}}}}
\newcommand{\continuum}{\ensuremath{\mathfrak{c}}}

\newcommand{\diam}[1]{\ensuremath{\sdiam{\left({#1}\right)}}}

\newcommand{\ep}{\varepsilon}

\newcommand{\ind}[1]{\ensuremath{\mbox{\boldmath{$1$}}_{#1}}}

\newcommand{\lint}[4]{\ensuremath{\int_{#1}^{#2}{#3}\:\mathrm{d}{#4}}}

\newcommand{\map}[3]{\ensuremath{{#1}:{#2}\to{#3}}}

\newcommand{\n}[1]{\ensuremath{\left\|{#1}\right\|}}
\newcommand{\ndot}{\ensuremath{\left\|\cdot\right\|}}

\newcommand{\ord}[2]{\ensuremath{\ordstart({#1},{#2})}}

\newcommand{\pow}[1]{\ensuremath\mathbf{P}({#1})}

\newcommand{\pn}[2]{\ensuremath{\left\|{#1}\right\|_{#2}}}
\newcommand{\pndot}[1]{\ensuremath{\left\|\cdot\right\|_{#1}}}

\newcommand{\ptndot}[1]{\ensuremath{\tri\cdot\tri_{#1}}}
\newcommand{\Q}{\mathbb{Q}}
\newcommand{\R}{\mathbb{R}}

\newcommand{\set}[2]{\ensuremath{\left\{{#1}\;:\;\,{#2}\right\}}}

\newcommand{\st}{($*$)}

\newcommand{\tri}{{\displaystyle |\kern-.9pt|\kern-.9pt|}}

\newcommand{\tn}[1]{\ensuremath{\left|\kern-.9pt\left|\kern-.9pt\left|{#1}\right|\kern-.9pt\right|\kern-.9pt\right|}}
\newcommand{\tndot}{\ensuremath{\left|\kern-.9pt\left|\kern-.9pt\left|\cdot\right|\kern-.9pt\right|\kern-.9pt\right|}}

\newcommand{\wone}{\ensuremath{\omega_1}}

\DeclareMathOperator{\aspan}{span}

\DeclareMathOperator{\card}{card}

\DeclareMathOperator{\ordstart}{ord}

\DeclareMathOperator{\sconv}{sconv}
\DeclareMathOperator{\sdiam}{diam}

\renewcommand{\geq}{\geqslant}
\renewcommand{\leq}{\leqslant}

\numberwithin{equation}{section}

\allowdisplaybreaks

\title[Gruenhage spaces and renorming theory]{Gruenhage spaces and their influence on Banach space renorming theory}
\author[R. J. Smith]{Richard J.~Smith}
\address[R. J. Smith]{School of Mathematics and Statistics, University College Dublin, Belfield, Dublin 4, Ireland}
\email{richard.smith@maths.ucd.ie}

\date{\today}

\begin{document}
\begin{abstract}
	In a paper from 1987, Gruenhage defined a class of topological spaces that now bear his name, and used it to solve a problem of Talagrand on the existence of dense $G_\delta$ metrizable subsets of Gul'ko compact spaces. Gruenhage's paper became highly influential among researchers in renorming theory, a branch of Banach space theory. In this paper we survey Gruenhage and related spaces, and their interactions with renorming theory.
\end{abstract}

\dedicatory{This article is dedicated to the memory of Prof.~Gary F.~Gruenhage (1947 -- 2023).}
	
\maketitle
	
\section{Introduction}\label{sect_intro}

Let $X$ be a topological space. In this paper we shall consider only Hausdorff spaces. We call a family $\mathscr{U}$ of subsets of $X$ \emph{$T_0$-separating} or simply \emph{separating} if, given distinct $x,y \in X$, there exists $U \in \mathscr{U}$ such that $\{x,y\} \cap U$ is a singleton (so in general it is not possible to choose $U$ so that $\{x,y\} \cap U=\{x\}$). Given a family $\mathscr{U}$ of subsets of $X$ and $x \in X$, we write
	\[
	 \ord{x}{\mathscr{U}} = \card\set{U \in \mathscr{U}}{x \in U}.
	\]

In his brief yet influential 1987 study of Gul'ko compact spaces, Gruenhage defined the following concept. 
	
\begin{definition}[cf.~\cite{gruenhage:87}*{p.~372}]\label{df:gruenhage_defn} A separating family $\mathscr{U}$ of $X$ is called \emph{$\sigma$-distributively point finite} if we can write $\mathscr{U}=\bigcup_{n \in \omega} \mathscr{U}_n$ in such a way that, given distinct $x,y \in X$
\begin{enumerate}[label={\upshape{(\roman*)}}]
	\item $\{x,y\} \cap U$ is a singleton for some $n \in \omega$ and $U \in \mathscr{U}_n$;
	 \item $\ord{x}{\mathscr{U}_n}$ or $\ord{y}{\mathscr{U}_n}$ is finite.
\end{enumerate}
\end{definition}

The motivation behind this definition will be explained in \Cref{sect:prelim}. Having isolated this property, Gruenhage answered positively a question first posed by Talagrand \cite{talagrand:79}*{p.~430} and asked again in \cite{comfort:negrepontis:82}*{p.~206} and \cite{negrepontis:84}*{p.~1111}. It was proved in \cite{namioka:74}*{Corollary 4.2} and later in \cite{benyamini:rudin:wage:77}*{Theorem 4.3} that every Eberlein compact space admits a dense $G_\delta$ metrizable subspace. Talagrand asked whether every Gul'ko compact space also enjoyed this property. Given the appropriate characterization of Gul'ko compact spaces, it is straightforward to show that every such space admits a family as in \Cref{df:gruenhage_defn} (see \Cref{th:gulko} and remarks thereafter). With this in mind, Gruenhage could give a positive answer to Talagrand's question. 
	
\begin{theorem}[{\cite{gruenhage:87}*{Theorem 1}}]\label{th:gruenhage}
Every compact space which admits a $\sigma$-distributively point-finite separating open cover has a dense $G_\delta$ metrizable subspace.
\end{theorem}

Clearly, if $\mathscr{U}$ is separating then $X\setminus\bigcup\mathscr{U}$ is a singleton at most, so we can easily replace the open cover in \Cref{th:gruenhage} by a family of open sets.
	
We note that a little before this result was published, independently, Leiderman showed that every Gul'ko compact space admits a dense metrizable subspace \cite{leiderman:85}. Also, very soon after \Cref{th:gruenhage} appeared, it was generalized to the much wider class of compact \emph{fragmentable} spaces (see \Cref{th:fragmentable} \ref{th:fragmentability_dense_Gdelta} below). Nevertheless, the idea of $\sigma$-distributively point finite families found other applications, notably forming a key catalyst in a perhaps unexpected area, namely the renorming theory of Banach spaces and specifically in the construction of so-called \emph{strictly convex dual norms}.

\begin{definition}\label{df:gruenhage_space}
We shall call a topological space $X$ \emph{Gruenhage} if it admits a family of open sets satisfying \Cref{df:gruenhage_defn}. If $X$ is in addition compact we shall call it \emph{Gruenhage compact}. 
\end{definition}

In this paper, we chart the development of these ideas over recent decades. It is designed to be reasonably self-contained and accessible to a specialist in topology. In \Cref{sect:prelim} we provide the necessary Banach space background and context, highlighting along the way some of the deep interactions between topology and functional analysis (which led to the isolation and study of Eberlein and Gul'ko compact spaces mentioned above), and touching on how topology came to play an increasingly important role in renorming theory. A background in functional analysis of the level of a first graduate course is assumed.

For further context and motivation, we provide in \Cref{sect:descriptive} a short account of so-called descriptive compact spaces, which form a class in between those of Gul'ko and Gruenhage compact spaces. In \Cref{sect:top_properties,sect:funct_anal}, we present most of the known topological and functional-analytic properties of Gruenhage spaces. In \Cref{sect:descendents}, we cover some generalisations and descendents of Gruenhage spaces, and conclude with some open problems.

In 2011, I had the pleasure of meeting Prof.~Gruenhage at the 26th Summer Conference on Topology and its Applications, at City College of CUNY. We had a brief discussion about Gruenhage spaces, how they had influenced renorming theory, and more recent developments, which we continued in some email exchanges thereafter. I was very happy to see that he saw fit to include some of those more recent developments in the third volume of Recent Progress in General Topology \cite{gruenhage:14}*{Section 5}. One could regard \Cref{sect:descendents} as an expanded version of this section that includes some even more recent updates.

I have included certain proofs of results either in the interests of keeping the survey self-contained, or because I saw an opportunity to improve the originals or formulate them in a way that relied less on heavy-duty techniques from functional analysis, and thus make them more amenable to specialists in topology.

Finally, we note that some of the topics addressed in this survey are also treated in a Banach space context (in different ways) in \cites{fabian:montesinos:zizler:10,smith:troyanski:10}. For a treatment of these topics from the perspective of linearly ordered topological spaces and generalized ordered spaces, we refer the reader to \cite{bennett:lutzer:13}. For more recent results on finding dense metrizable subspaces of compact spaces, using very different approaches, please see \cites{leiderman:spadaro:todorcevic:22,todorcevic:24}.
	
\section{Banach spaces, topology and renorming theory}\label{sect:prelim}

For general information about Banach space theory, we refer the reader to \cite{fabian:11}, with elements of Chapters 3, 7, 8, 13 and 14 being of particular relevance here. For further information on the interactions between classes of Banach spaces and associated compact spaces (particularly Eberlein and Gul'ko compact spaces), we recommend \cite{fabian:97}. All undefined concepts and proofs of unattributed statements in this survey can be found in these books. All Banach spaces in this paper will be considered over $\R$. One learns in a first undergraduate course in functional analysis that the closed unit ball $B_X$ of a Banach space $X$ is compact in the norm topology if and only if $\dim X < \infty$. Much of functional analysis has to make do with this lack of compactness, relying instead on completeness or some additional structure such as separability or reflexivity or, more generally, compactness of a large enough subset of $X$ with respect to a locally convex topology on $X$. In this context, being ``large enough'' means the following. We shall say that $X$ is \emph{generated} by $K \subseteq X$ if $X$ is the norm-closed linear span of $K$, written $X=\cl{\aspan}^{\ndot}(K)$. Observe that a Banach space is separable if and only if it is generated by a norm-compact subset.  

Given a subset $F$ of the dual space $X^*$, we define $\sigma(X,F)$ to be the weakest topology on $X$ with respect to which every element of $F$ is continuous. This topology is Hausdorff if and only if $F$ is \emph{separating}, that is, given $x \in X$, there exists $f \in F$ satisfying $f(x) \neq 0$; hereafter we assume all such subsets $F$ have this property. Recall that there exists a canonical linear isometric embedding $J:X \to X^{**}$ given by $(Jx)(f)=f(x)$, $x \in X$, $f \in X^*$; usually we drop $J$ and simply regard $X$ as a subspace of its bidual $X^{**}$. With the above in mind, the \emph{weak} topology on $X$ and the \emph{weak}$^*$ topology on $X^*$ are given by $w=\sigma(X,X^*)$ and $w^* = \sigma(X^*,X)$, respectively. 

Given a compact space $K$, we denote by $C(K)$ the Banach space of continuous real-valued functions on $K$, equipped with the supremum norm $\pndot{\infty}$. By the Riesz representation theorem, we identify the dual space $C(K)^*$ with the space of signed Radon measures on $K$, equipped with the total variation norm $\pndot{1}$, which satisfies
\begin{equation}\label{eq:variation}
 \pn{\mu}{1} = \sup\set{\lint{K}{}{f}{\mu}}{f \in B_{C(K)}}.
\end{equation}
Given $\mu \in C(K)^*$, we denote by $|\mu|, \mu^+$ and $\mu^-$ the total variation and positive and negative parts of $\mu$, respectively.

Being weaker topologies, a given subset of $X$ or $X^*$ has a better chance of being compact with respect to $w$ or $w^*$, respectively. The unit ball $B_X$ is $w$-compact if and only if $X$ if $X$ is \emph{reflexive}, i.e.~the map $J$ above is surjective.

Regarding $w^*$-compactness, the Banach-Alaoglu theorem states that the dual unit ball $B_{X^*}$ is \emph{always} $w^*$-compact. Because (by definition) the elements of $X \subseteq X^{**}$ (and only these) are $w^*$-continuous when regarded as functionals on $X^*$, we obtain another canonical linear isometric embedding $T:X \to C(B_{X^*},w^*)$, defined again by evaluation: $Tx(f)=f(x)$, $x \in X$, $f \in B_{X^*}$.

This is the starting point of a rich theory that traces the deep connections between the structural and geometric properties of $X$ and the topological properties of $(B_{X^*},w^*)$. For example, $X$ is separable if and only if $(B_{X^*},w^*)$ is metrizable, and a compact space $K$ is metrizable if and only if $C(K)$ is separable. We shall mention some progressively larger classes of Banach spaces that include the separable and reflexive ones, together with their related classes of compact spaces.

We say that $X$ is \emph{weakly compactly generated} (WCG) if it is generated by a $w$-compact subset. Evidently all separable and reflexive spaces are weakly compactly generated, but there is a canonical counterexample to the converse of this statement. Given a set $\Gamma$, let $c_0(\Gamma)$ denote the Banach space of all functions $x:\Gamma \to \R$ such that, given $\ep>0$, there exists a finite set $F \subseteq \Gamma$ such that $|x(\gamma)| < \ep$ whenever $\gamma\in\Gamma\setminus F$. We equip $c_0(\Gamma)$ with the supremum norm $\pndot{\infty}$. Given $\gamma \in \Gamma$, define the standard unit vector $e_\gamma:\Gamma\to\R$ by $e_\gamma(\gamma')=1$ if $\gamma'=\gamma$, and $e_\gamma(\gamma')=0$ otherwise. The set $K:=\{0\} \cup \set{e_\gamma}{\gamma \in \Gamma}$ is $w$-compact because its only $w$-accumulation point is $0$, and evidently $c_0(\Gamma)=\cl{\aspan}^{\pndot{\infty}}(K)$. However, $c_0(\Gamma)$ is non-separable whenever $\Gamma$ is uncountable, and non-reflexive whenever $\Gamma$ is infinite.

In a sense, $c_0(\Gamma)$ is the archetypal WCG space. If $X$ is WCG then there exists a set $\Gamma$ and a bounded linear injection $T:X\to c_0(\Gamma)$. In particular, it follows that $K$ must be (linearly) homeomorphic to the $w$-compact set $T(K) \subseteq c_0(\Gamma)$. We say that a compact space $K$ is \emph{Eberlein compact} if it is homeomorphic to a $w$-compact subset of a Banach space or, equivalently, a $w$-compact subset of $c_0(\Gamma)$ for some set $\Gamma$. A Banach space $X$ is a subspace of a WCG space if and only if $(B_{X^*},w^*)$ is Eberlein compact (the class of WCG spaces is not hereditary:~if $Y$ is a closed linear subspace of a WCG space $X$, then $Y$ need not be WCG), and $C(K)$ is WCG if and only if the compact space $K$ is Eberlein.

More generally, a Banach space $X$ is called \emph{weakly countably determined} (WCD) or a \emph{Va\v s\'ak space} if there are $w^*$-compact subsets $K_n$, $n \in \omega$, of the bidual $X^{**}$ such that for every $x \in X$ and $\xi \in X^{**}\setminus X$, there exists $n \in \omega$ satisfying $\{x,\xi\} \cap K_n = \{x\}$. Every WCG space $X$ is WCD because, if $X=\cl{\aspan}^{\ndot}(K)$ for some $w$-compact $K \subseteq X$, it can be shown that the $w^*$-compact subsets $K_{n,m}:= nK + 2^{-m}B_{X^{**}}$ of $X^{**}$, $m,n \in \omega$, fulfil the above requirement. It is relatively straightforward to show that (unlike WCG spaces) the class of WCD spaces is hereditary. A compact space $K$ is said to be \emph{Gul'ko compact} if $C(K)$ is WCD. A Banach space $X$ is WCD if and only if $(B_{X^*},w^*)$ is Gul'ko compact.

To see the motivation behind \Cref{df:gruenhage_defn}, we present the following purely topological characterizations of Eberlein and Gul'ko compact spaces. Recall that a family $\mathscr{U}$ of sets in a topological space is \emph{point finite} if $\ord{x}{\mathscr{U}}$ is finite for all $x \in X$, and $\mathscr{U}$ is \emph{$\sigma$-point finite} if it can be written as $\mathscr{U}=\bigcup_{n \in \omega}\mathscr{U}_n$, where each $\mathscr{U}_n$ is point finite.
	
\begin{theorem}[\cite{rosenthal:74}*{Theorem 3.1}]\label{th:rosenthal}
A compact space is Eberlein if and only if it admits a $\sigma$-point finite separating family consisting of open $F_\sigma$ sets. 
\end{theorem}
	
By weakening the $\sigma$-point finiteness property we can obtain an analogous characterization of Gul'ko compact spaces. Following \cite{sokolov:84}*{Definition 3}, we say that a family $\mathscr{U}$ of subsets of $X$ is \emph{weakly $\sigma$-point-finite} if we can write $\mathscr{U}=\bigcup_{n \in \omega} \mathscr{U}_n$ with the property that, given $x \in X$ and $U \in \mathscr{U}$, there exists $n \in \omega$ satisfying $U \in \mathscr{U}_n$ and $\ord{x}{U_n}$ is finite.

Sokolov used this concept to provide a characterization of Gul'ko compact spaces that closely resembles \Cref{th:rosenthal}.
	
\begin{theorem}[{\cite{sokolov:84}*{Theorem 8}}]\label{th:gulko}
 A compact space is Gul'ko if and only if it admits a weakly $\sigma$-point-finite separating cover consisting of open $F_\sigma$ sets.
\end{theorem}

(We note that another topological characterization of Gul'ko compact spaces is given in \cite{mercourakis:87}*{Theorem 3.3}.) Every separating weakly $\sigma$-point-finite family $\mathscr{U}=\bigcup_{n \in \omega}\mathscr{U}_n$ is $\sigma$-distributively point finite. Indeed, given distinct $x,y \in X$, as $\mathscr{U}$ is separating there exists $U \in \mathscr{U}$ having the property that $\{x,y\} \cap U$ is a singleton. If $x \in U$ then there exists $n \in \omega$ such that $U \in \mathscr{U}_n$ and $\ord{x}{U_n}$ is finite; we reach a similar conclusion if $y \in U$. Therefore by Theorem \ref{th:gulko}, every Gul'ko compact space is Gruenhage \cite{gruenhage:87}*{p.~372}. 

Concerning the classes of compact spaces mentioned in this survey, we have the following chain of implications: 
\begin{align*}
 \text{metrizable }\Rightarrow\text{ Eberlein }\Rightarrow\text{ Gul'ko }&\Rightarrow\text{ descriptive }\\
 &\Rightarrow\text{ Gruenhage }\Rightarrow\text{ {\st} } \Rightarrow\text{ fragmentable.}
\end{align*}
We define descriptive compact spaces, fragmentable spaces and spaces with property {\st} in later sections. The spaces in these classes could be said to be \emph{generalised metrizable spaces}, in the sense that they retain (increasingly flimsy) vestiges of metrizability (concerning this wide topic, the reader is encouraged to consult Gruenhage's authoritative surveys, e.g.~\cites{gruenhage:84,gruenhage:14}). None of the implications above are reversible. The $1$-point compactification of an uncountable discrete space is Eberlein and non-metrizable, and the first example of a Gul'ko, non-Eberlein compact space is shown in \cite{talagrand:79}*{Th\'eor\`eme 4.3}. The irreversibility of the other implications is demonstrated in later sections.

We address renorming theory in the final part of the section. In this theory, one attempts to replace the given canonical norm on a Banach space by an equivalent norm that possesses some superior geometric property, such as some degree of smoothness or strict convexity. Whether or not this is possible will depend on the structure of the given Banach space. Chapter VII, \S 4 of M.~Day's book \cite{day:73} is entitled `Isomorphisms to improve the norm', which is a good way of putting it. We refer the reader to \cite{deville:godefroy:zizler:93} and the more up-to-date \cite{guirao:montesinos:zizler:22} for very comprehensive treatments of this theory. Here, we will focus almost exclusively on the following classical geometric property of norms.

\begin{definition}[\cite{clarkson:36}*{p.~404}]\label{df:strictly_convex}
 A norm $\ndot$ on a Banach space $X$ is called \emph{strictly convex} (or \emph{rotund}) if $x=y$ whenever $\n{x}=\n{y}=\n{\frac{1}{2}(x+y)}$. 
\end{definition}

Equivalently, the norm $\ndot$ is strictly convex if the unit sphere of $X$ with respect to $\ndot$ contains no non-trivial straight line segments, i.e.~if $x \neq y$, $\n{x}=\n{y}=1$ and $t \in (0,1)$, then $\n{tx+(1-t)y} < 1$. There are a number of stronger cousins of strict convexity, such as uniform rotundity or local uniform rotundity, which we touch on only very briefly. By the parallelogram identity, it follows that the canonical norm on Hilbert space is uniformly rotund and thus strictly convex. In \cite{lindenstrauss:76}*{Q.18}, Lindenstrauss asked for a characterization of those Banach spaces which have an equivalent strictly convex norm. Much of the research presented in this survey was motivated by this question, in particular with regard to strictly convex \emph{dual} norms on dual Banach spaces. We call an equivalent norm $\tndot$ on a dual space $X^*$ a \emph{dual} norm if there exists an equivalent norm on $X$ (again denoted $\tndot$) whose dual norm is $\tndot$. One incentive for looking for equivalent strictly convex dual norms is that their predual norms are \emph{G\^ateaux smooth} \cite{fabian:11}*{Corollary 7.23 (i)}. In addition, confining our attention to dual spaces makes life easier because we have the $w^*$-compactness of the dual unit balls at our disposal. 

In 1987, Mercourakis proved that if $X$ is WCD then $X^*$ admits an equivalent strictly convex dual norm \cite{mercourakis:87}*{Theorem 4.6}. This result was obtained using what we can now call classical functional-analytic renorming techniques. In more recent years, starting with the Murcia-Valencia School, mathematicians recognised the importance of bringing to bear purely topological techniques in renorming theory. Let $X$ be a Banach space and $\tau$ a locally convex topology on $X$ (e.g. the norm or $w$-topology, or the $w^*$-topology if $X$ is a dual space). We say that a norm $\ndot$ on $X$ is \emph{$\tau$-locally uniformly rotund} ($\tau$-LUR) if, given $x \in X$ and $x_n \in X$, $n \in \omega$, satisfying $\n{x}=\n{x_n}=1$ for all $n \in \omega$ and $\lim_n \n{x+x_n}=2$, we must have $\tau$-$\lim_n x_n = x$ (cf.~\cite{lovaglia:55}*{Definition 0.2}). Such norms are always strictly convex:~given $x,y \in X$ satisfying $\n{x}=\n{y}=\n{\frac{1}{2}(x+y)}=1$, simply set $x_n=y$ for all $n \in \omega$. One of the early successes of the topological approach to renorming theory was to show that if $X$ admitted an equivalent $w$-LUR norm, then $X$ admitted an equivalent $\ndot$-LUR (i.e.~LUR) norm \cite{molto:99}*{Main Theorem}.

We finish the section with some specific results about constructing norms that will be used in proofs later on. The first one is a key tool for identifying dual norms.

\begin{lemma}[\cite{fabian:11}*{Lemma 3.97}]\label{pr:lsc}
 Let $\tndot$ be an equivalent norm on the dual Banach space $X^*$. Then $\tndot$ is a dual norm if and only if it is $w^*$-lower semicontinuous.
\end{lemma}
	
\begin{proof}
If $\tndot$ is a dual norm then (with $\tndot$ also denoting the predual norm), we have
\[
 \tn{f} = \sup\set{f(x)}{x \in X,\,\tn{x} \leq 1}, \quad f \in X^*.
\]
Thus $\tndot$ is the supremum of a family of $w^*$-continuous functions $f \mapsto f(x)$, $\tn{x} \leq 1$, on $X^*$, so it is $w^*$-lower semicontinuous. Conversely, let $\tndot$ be $w^*$-lower semicontinuous. Then $B:=\set{f \in X^*}{\tn{f}\leq 1}$ is $w^*$-closed. The norm $|\cdot|$ on $X$ defined by
\[
 |x| = \sup\set{f(x)}{f \in B}, \quad x \in X,
\]
can be seen to be equivalent to the given norm on $X$, and its dual norm satisfies $|\cdot| \leq \tndot$ on $X^*$. Moreover, we have equality. Let $f \in X^*$ satisfy $\tn{f} > 1$, i.e.~$f \notin B$. Because $B$ is convex and $w^*$-closed, by the (topological) Hahn-Banach separation theorem \cite{fabian:11}*{Theorem 3.32}, there exists $x \in X$ and $\alpha \in \R$ such that $f(x) > \alpha \geq g(x)$ for all $g \in B$. Thus $0 < |x| \leq \alpha$ and $|f| \geq f(x/\alpha) > 1$. From this we quickly obtain equality, meaning that $\tndot$ is a dual norm.
\end{proof}

We observe the following useful inequality that holds for any seminorm $\ndot$ on a Banach space:
\begin{equation}\label{eq:parallelogram}
2\n{x}^2 + 2\n{y}^2 - \n{x+y}^2 \geq 2\n{x}^2 + 2\n{y}^2 - (\n{x}+\n{y})^2 = (\n{x}-\n{y})^2 \geq 0. 
\end{equation}

The next result is a technical lemma often used in the construction of strictly convex norms (and stronger variants).

\begin{lemma}[cf.~\cite{fabian:11}*{Fact 7.7 and Theorem 8.1}]\label{lm:parallelogram}
Suppose that $\pndot{n}$ is a sequence of seminorms on a Banach space $X$, such that the formula
\[
 \tn{x}^2 = \sum_{n = 0}^\infty \pn{x}{n}^2, \quad x \in X,
\]
defines an equivalent norm $\tndot$ on $X$. Then $\tn{x}=\tn{y}=\tn{\frac{1}{2}(x+y)}$ implies $\pn{x}{n}=\pn{y}{n}=\pn{\frac{1}{2}(x+y)}{n}$ for all $n \in \omega$.
\end{lemma}

Note that the Cauchy-Schwarz inequality can be used to prove that $\tndot$ above satisfies the triangle inequality.

\begin{proof}
Let $\tn{x}=\tn{y}=\tn{\frac{1}{2}(x+y)}$. Then
\begin{align*}
 0 = 2\tn{x}^2 + 2\tn{y}^2 - \tn{x+y}^2 = \sum_{n=0}^\infty 2\pn{x}{n}^2 + 2\pn{y}{n}^2 - \pn{x+y}{n}^2
\end{align*}
implies $2\pn{x}{n}^2 + 2\pn{y}{n}^2 - \pn{x+y}{n}^2=0$ for all $n \in \omega$ by \eqref{eq:parallelogram}, whence $\pn{x}{n}=\pn{y}{n}$ for all $n$, again by \eqref{eq:parallelogram}. The conclusion follows.
\end{proof}

Let $\mathbf{P}$ be some geometric property of norms, such as strict convexity. It is necessary to point out that constructing a equivalent norm on $X^*$ satisfying $\mathbf{P}$ is often possible, while constructing a equivalent \emph{dual} norm on $X^*$ satisfying $\mathbf{P}$ might not be. 
For example, it has long been known that $C(\wone+1)^*$, where $\wone+1$ is endowed with its usual scattered compact interval topology, admits no equivalent strictly convex dual norm \cite{talagrand:86}*{Th\'eor\`eme 3} (see the discussion after \Cref{th:star_properties} for a quick proof of this). However, being scattered, by a result of W.~Rudin $C(\wone+1)^*$ is linearly isometric to $\ell_1(\wone+1)$ \cite{fabian:11}*{Theorem 14.24}. In turn, $\ell_1(\wone+1)$ embeds naturally into the Hilbert space $\ell_2(\wone+1)$. Hence we can define an equivalent norm $\ndot$ on $C(\wone+1)^*$ by setting $\n{\mu}^2=\pn{\mu}{1}^2+\pn{\mu}{2}^2$, $\mu \in \ell_1(\wone+1)$, where $\pndot{2}$ denotes the Hilbert space norm. By \Cref{lm:parallelogram} and the fact that Hilbert space norms are strictly convex, we see that $\ndot$ is also strictly convex.
	
Finally, we give an example of what is known as a \emph{transfer technique} in renorming theory. The following result is stated in much greater generality in \cite{deville:godefroy:zizler:93}.

\begin{proposition}[cf.~\cite{deville:godefroy:zizler:93}*{Theorem 2.1}]\label{th:transfer}
 Let a dual operator $T^*:Y^*\to X^*$ have dense range and let $Y^*$ admit an equivalent strictly convex dual norm. Then $X^*$ admits an equivalent strictly convex dual norm.
\end{proposition}

\section{Descriptive compact spaces}\label{sect:descriptive}

Before we begin our tour of the properties of Gruenhage spaces, to provide further context and motivation, it is worth introducing the class of \emph{descriptive compact spaces}. The study of descriptive compact spaces was initiated in \cite{raja:03} and developed further in \cite{oncina:raja:04}; for the most part, we follow \cite{raja:03} in this section. Recall that a family $\mathscr{H}$ of pairwise disjoint subsets of $X$ is \emph{discrete} if, given $x \in X$, there exists an open set $U \ni x$ having the property that $U$ meets at most one member of $\mathscr{H}$. Of course, discrete families play a central role in metrizability theory and the general metrization theorem. However, they are a little too restrictive for our purposes. We say that the family $\mathscr{H}$ is \emph{isolated}, or \emph{relatively discrete}, if, given $x \in H \in \bigcup{\mathscr{H}}$, there exists an open set $U \ni x$ which meets no members of $\mathscr{H}$ except $H$; i.e.~$\mathscr{H}$ is isolated if it is discrete with respect to the subspace topology of $\bigcup \mathscr{H}$. If a family can be written as the union of countably many isolated subfamilies then we call it \emph{$\sigma$-isolated}. We say that the family $\mathscr{H}$ is a \emph{network} for $X$ if, given an open subset $U$ of $X$ and $x \in U$, we have $x \in H \subseteq U$ for some $H \in \mathscr{H}$. Fundamental facts about $\sigma$-isolated families and networks, and their role in Banach space theory, are given in \cite{hansell:01}. Compare the next definition with the general metrization theorem.

\begin{definition}[{\cite{raja:03}*{Definition 1.1}}]\label{df:descriptive}
A compact space $K$ is called \emph{descriptive} if it possesses a $\sigma$-isolated network.
\end{definition}

Using another result of Gruenhage, namely that Gul'ko compact spaces are so-called \emph{hereditary weakly-$\theta$-refinable} \cite{gruenhage:87}*{Theorem 2}, Raja proves that all Gul'ko compact spaces are descriptive \cite{raja:03}*{Corollary 2.4}. We gather some more examples of descriptive compact spaces. We shall say that a topological space $X$ is \emph{$\sigma$-discrete} if we can write $X=\bigcup_{n \in \omega} X_n$, where each $X_n$ is discrete in its subspace topology. A Baire Category argument shows that every $\sigma$-discrete compact space must be scattered. All scattered compact spaces having countable Cantor-Bendixson rank are $\sigma$-discrete. On the other hand it is straightforward to construct examples of $\sigma$-discrete compact spaces having arbitrarily large rank. It is clear that every compact $\sigma$-discrete space is descriptive. This simple observation allows us to separate the classes of Gul'ko and descriptive compact spaces. Indeed, every compact space $K$ having Cantor-Bendixson rank $3$ is descriptive, however, there exist plenty of such spaces that are separable yet non-metrizable, and no space of this form can be Gul'ko (this follows because every Gul'ko compact space is Corson compact \cite{gulko:79}, and separable Corson compact spaces are easily seen to be metrizable). The standard examples of these, known in the literature as Alexandrov-Urysohn spaces or Mr\'owka-Isbell spaces, are constructed using uncountable almost disjoint families of subsets of $\omega$, and make up an extremely versatile class \cite{hernandez:hrusak:18}.

We note that an example of a Gruenhage, non-Gul'ko compact space was given in \cite{argyros:mercourakis:93}*{Theorem 3.3}; this space was shown to be descriptive in \cite{oncina:raja:04}*{Example 4.5} and later, using a different method, in \cite{smith:troyanski:09}*{Example 1}.

If we confine our attention to scattered spaces then we obtain an equivalence between descriptive and $\sigma$-discrete compact spaces. The roots of the argument in the proof below can be traced to \cite{raja:02}*{Theorem 2.4} and, as in \cite{smith:troyanski:10}, we attribute the result to Raja.
	
\begin{proposition}[Raja \cite{smith:troyanski:10}*{Corollary 4}]\label{df:descriptive_scattered}
 Let $K$ be a scattered descriptive compact space. Then $K$ is $\sigma$-discrete.
\end{proposition}

\begin{proof}
 Because $K$ is scattered, we can use transfinite recursion to write $K=\set{x_\xi}{\xi \in \lambda}$ for some ordinal $\lambda$, where the $x_\xi$ are distinct and each set $U_\alpha:=\set{x_\xi}{\xi \in \alpha+1}$, $\alpha \in \lambda$, is open in $K$. Let $\mathscr{H}=\bigcup_{n \in \omega} \mathscr{H}_n$ be a network for $K$, where each $\mathscr{H}_n$ is an isolated family of sets. Given non-empty $E \subseteq K$, let $\Lambda_E = \set{\alpha \in \lambda}{x_\alpha \in E \subseteq U_\alpha}$ and set $\alpha_E = \min \Lambda_E$ whenever $\Lambda_E \neq \varnothing$. Given $n \in \omega$, define
 \[
  D_n = \set{x_{\alpha_H}}{H \in \mathscr{H}_n \text{ and }\Lambda_H \neq \varnothing}.
 \]
 Because each $\mathscr{H}_n$ is isolated, each $D_n$ is discrete in its subspace topology. Given $\alpha \in \lambda$, as $\mathscr{H}$ is a network, there exists $n \in \omega$ and $H \in \mathscr{H}_n$ such that $x_\alpha \in H \subseteq U_\alpha$. Thus $\alpha \in \Lambda_H$, and as $x_\alpha \notin U_\xi$ for all $\xi \in \alpha$, we obtain $\alpha=\alpha_H$ and $x_\alpha \in D_n$. Therefore $K$ is $\sigma$-discrete.
\end{proof}

The particular application of transfinite recursion above is a simplified version of the argument in Ribarska's characterization of fragmentable spaces -- see \Cref{th:fragmentability} below.

The sets in an arbitrary network can be very wild. However, given quite general circumstances, it is possible (and desirable) to replace them with more well-behaved sets.

\begin{lemma}[cf.\ {\cite{raja:03}*{Lemma 2.1}}]\label{lm:manicure}
Let $X$ be a regular topological space and let $\mathscr{H} = \bigcup_{n \in \omega} \mathscr{H}_n$ be a network for $X$, where
each $\mathscr{H}_n$ is isolated. Given $n \in \omega$, set $A_n = \cl{\bigcup\mathscr{H}_n}$ and, for every $H \in \mathscr{H}_n$, define the open set
\[
U_H = X\setminus \cl{\bigcup(\mathscr{H}_n\setminus\{H\})}.
\]
Then the families given by
\[
\mathscr{G}_n = \set{A_n \cap U_H}{H \in \mathscr{H}}, \quad n\in\omega,
\]
are isolated, and $\mathscr{G}:=\bigcup_{n\in\omega}\mathscr{G}_n$ is a network for $X$.
\end{lemma}

\begin{proof}
As $X$ is regular, $\set{\cl{H}}{H \in \mathscr{H}}$ is also a network for $X$. Given $H \in \mathscr{H}_n$, observe that $H \subseteq A_n \cap U_H \subseteq \cl{H}$. Thus $\mathscr{G}$ is again a network. Moreover, since $U_H \cap L = \varnothing$ whenever $L \in \mathscr{H}_n\setminus\{H\}$, the elements of each $\mathscr{G}_n$ must be pairwise disjoint. Therefore each $\mathscr{G}_n$ is isolated.
\end{proof}

Recall the definition of $\tau$-LUR norms and \cite{molto:99}*{Main Theorem} from \Cref{sect:prelim}. Another of the most striking connections between topology and renorming theory is provided by the next result.  

\begin{theorem}[cf.~\cite{raja:03}*{Theorem 1.3}]\label{th:raja} The following statements hold.
\begin{enumerate}[label={\upshape{(\roman*)}}]
\item\label{th:raja_1} The dual space $X^*$ admits an equivalent $w^*$-LUR dual norm if and only if $(B_{X^*},w^*)$ is descriptive.
\item\label{th:raja_2} A compact space $K$ is descriptive if and only if $C(K)^*$ admits an equivalent $w^*$-LUR dual norm.
\end{enumerate}
\end{theorem}

In particular, \Cref{th:raja} \ref{th:raja_1} shows that an ostensibly geometric property of dual Banach spaces, namely the existence of an equivalent $w^*$-LUR norm, is really a non-linear topological property. It is natural to wonder if it is possible to isolate another topological property which would yield analogues of \Cref{th:raja} \ref{th:raja_1} or \ref{th:raja_2} (or both), where ``$w^*$-LUR'' is replaced by ``strictly convex''. This is another question that motivated much of the work presented in the following sections.
	
\section{Topological properties of Gruenhage spaces}\label{sect:top_properties}
	
To explore the topological properties of Gruenhage spaces, we present two more user-friendly equivalent formulations of the families described in \Cref{df:gruenhage_defn}.
	
\begin{proposition}[{\cite{stegall:91}*{Proposition 7.4} and \cite{smith:09}*{Proposition 2}}]\label{pr:gruenhage_equiv}
Let $X$ be a topological space. The following statements are equivalent.
\begin{enumerate}[label={\upshape{(\roman*)}}]
\item\label{pr:gruenhage_equiv_1} $X$ is Gruenhage;
\item\label{pr:gruenhage_equiv_2} there exist closed sets $A_n \subseteq X$ and families $\mathscr{H}_n$ of pairwise disjoint sets open in $A_n$, $n \in \omega$, such that $\bigcup_{n \in \omega} \mathscr{H}_n$ separates points;
\item\label{pr:gruenhage_equiv_3} there exist open sets $R_n$ and families $\mathscr{U}_n$ of open subsets of $X$, $n \in \omega$, such that $R_n \subseteq U$ for all $U \in \mathscr{U}_n$, $U \cap V = R_n$ whenever $U,V \in \mathscr{U}_n$ are distinct, and $\bigcup_{n \in \omega} \mathscr{U}_n$ separates points.
\end{enumerate}
\end{proposition}

We should point out that the property given in \Cref{pr:gruenhage_equiv} \ref{pr:gruenhage_equiv_2} first appears essentially in \cite{stegall:91}*{Lemma 7.1 (i)}. The implication \ref{pr:gruenhage_equiv_1} $\Rightarrow$ \ref{pr:gruenhage_equiv_2} is contained in the proofs of \cite{stegall:91}*{Lemma 7.1 and Proposition 7.4}, and part of the arguments therein can be traced back to \cite{gruenhage:87}. The implications \ref{pr:gruenhage_equiv_2} $\Rightarrow$ \ref{pr:gruenhage_equiv_3} and \ref{pr:gruenhage_equiv_3} $\Rightarrow$ \ref{pr:gruenhage_equiv_1} in \Cref{pr:gruenhage_equiv} are trivial.
	
	\begin{proof}[Proof of \Cref{pr:gruenhage_equiv} \ref{pr:gruenhage_equiv_1} $\Rightarrow$ \ref{pr:gruenhage_equiv_2}] 
	Let $\mathscr{U} = \bigcup_{n \in \omega} \mathscr{U}_n$ be a family fulfilling the conditions of \Cref{df:gruenhage_defn}. Given $m,n \in \omega$, define the closed set
	\[
	 A_{m,n} = \set{x \in X}{\ord{x}{\mathscr{U}_n} \leq m},
	\]
    and given $m \geq 1$, define the family of pairwise disjoint sets
    \[
     \mathscr{H}_{m,n} = \set{V \cap A_{m,n}\setminus A_{m-1,n}}{V = U_1 \cap \ldots \cap U_m \text{ for some distinct }U_1,\ldots,U_m \in \mathscr{U}_n}.
    \]
Evidently each element of $\mathscr{H}_{m,n}$ is open in $A_{m,n}$ and, given the properties of $\mathscr{U}$, it is easy to check that the sets $A_{m,n}$ and families $\mathscr{H}_{m,n}$ satisfy the properties asked for in \ref{pr:gruenhage_equiv_2}.
\end{proof}
	
	More ways to describe Gruenhage spaces are given in \cite{garcia:padial:15}*{Theorem 1.1}. Given \Cref{lm:manicure} and \Cref{pr:gruenhage_equiv} \ref{pr:gruenhage_equiv_2}, the next result is immediate.
	
	\begin{corollary}[{\cite{smith:09}*{Corollary 4}}]\label{co:descriptive_Gruenhage}
	 Descriptive compact spaces are Gruenhage.
	\end{corollary}

	We shall use trees to show that there exist Gruenhage compact spaces that are not descriptive. Recall that a \emph{tree} $(\Upsilon,\preccurlyeq)$ is a partially ordered set such that, for each $x \in \Upsilon$, the set $\set{u \in \Upsilon}{u \prec x}$ of strict predecessors is well-ordered (so here we mean tree in the sense of combinatorial set theory, rather than descriptive set theory). With respect to the standard interval topology $\Upsilon$ is a scattered locally compact space (to ensure that $\Upsilon$ is Hausdorff, we assume that every non-empty chain in $\Upsilon$ has at most one minimal upper bound). Recall that $\Upsilon$ is said to be \emph{$L$-embeddable}, where $(L,\leq)$ is a totally ordered set, if there exists a function $f:\Upsilon\to L$ such that $f(x) < f(y)$ whenever $x \prec y$. A tree $\Upsilon$ is called \emph{special} if it can be written as the union of countably many antichains or, equivalently, if it is $\Q$-embeddable.
	
	Define $\sigma\Q = \set{x \subseteq \Q}{x \text{ is well-ordered and bounded above}}$, and partially order $\sigma\Q$ by end-extension:~$x \preccurlyeq y$ if and only if $x$ is an initial segment of $y$. By a classical result of Kurepa, $\sigma\Q$ is not $\Q$-embeddable and is thus non-special. Define $K$ to be the scattered $1$-point compactification $\sigma\Q \cup \{\infty\}$. Then $K$ is non-descriptive. Indeed, if it were then it would be $\sigma$-discrete by \Cref{df:descriptive_scattered}. However, if the $1$-point compactification $K = \Upsilon \cup \{\infty\}$ of a tree $\Upsilon$ is $\sigma$-discrete, then $\Upsilon$ must be special (cf.~\cite{smith:06}*{Theorem 4}).
	
	\begin{example}\label{ex:sigma_Q}
	 The space $K=\sigma\Q \cup \{\infty\}$ is Gruenhage. Indeed, set $U_\infty = \sigma\Q$ and, given $q \in \Q$, define $U_q = \set{x \in \sigma\Q}{q \in x}$. Each such $U_q$ is open. In particular, given $q \in \Q$ and $x \in U_q$, set $u=(-\infty,q) \cap x \in \sigma\Q$ and observe that the basic open interval $\set{v \in \sigma\Q}{u \prec v \preccurlyeq x}$ is included in $U_q$. Given $q \in \Q \cup \{\infty\}$, set $R_q=\varnothing$ and $\mathscr{U}_q = \{U_q\}$. Then it is easy to see that the $\mathscr{U}_q$ separate points of $K$, with $U_\infty$ separating $\infty$ from every point of $\sigma\Q$.
	\end{example}

	A complete characterization of the trees whose $1$-point compactifications are Gruenhage, which provides many more examples of Gruenhage, non-descriptive compact spaces, is given in \cite{smith:09}*{Corollary 17}.
	
	One cannot fail to notice that, in \Cref{ex:sigma_Q}, we only needed countably many open sets to separate points of $K$, rather than countably many families of open sets. This turns out to be true whenever $X$ is a Gruenhage space and $\card{X} \leq \continuum$; in this case \Cref{pr:gruenhage_equiv} \ref{pr:gruenhage_equiv_3} can be reduced to something strikingly simple.
	
	\begin{proposition}[{\cite{smith:troyanski:10}*{Proposition 2}}]\label{pr:gruenhage_continuum} Let $X$ be a topological space with $\card{X} \leq \continuum$. Then $X$ is Gruenhage if and only if there is a countable family of open subsets of $X$ that separates points.
	\end{proposition}

	\begin{proof}
	 One implication is trivial. Conversely, let sets $R_n$ and families $\mathscr{U}_n$, $n \in \omega$, satisfy the conditions of \Cref{pr:gruenhage_equiv} \ref{pr:gruenhage_equiv_3}. Since $\card{X} \leq \continuum$, it is clear that $\card\mathscr{U}_n \leq \continuum$ for each $n$. Fix injections $\map{\pi_n}{\mathscr{U}_n}{2^\omega}$ and define
	 \[
	 U_{n,k,i} = \bigcup\set{U \in \mathscr{U}_n}{\pi_n(U)(k) = i}, \quad n,k \in \omega,\, i \in 2.
	\]
	Then the $U_{n,k,i}$ separate points of $X$. Given distinct $x,y \in X$, let $n \in \omega$ such that $\{x,y\} \cap U$ is a singleton for some $U \in \mathscr{U}_n$. Without loss of generality assume $\{x,y\} \cap U = \{x\}$. If $y \notin \bigcup\mathscr{U}_n$ then $\{x,y\} \cap U_{n,0,\pi_n(0)} = \{x\}$. If $y \in \bigcup\mathscr{U}_n$ then $y \in V$ for some $V \in \mathscr{U}_n\setminus\{U\}$. Since $y \notin R_n$, we observe that $V$ is unique. As $U \neq V$, there exists $k \in \omega$ satisfying $i:=\pi_n(U)(k) \neq \pi_n(V)(k)$. Then $x \in U_{n,k,i}$ and the uniqueness of $V$ forces $y \notin U_{n,k,i}$.
	\end{proof}
	
	Of course, if a compact space $K$ admits a countable family $\mathscr{U}$ of open sets that can separate points in a $T_1$ sense (given $x,y \in K$, there exists $U\in\mathscr{U}$ such that $\{x,y\} \cap U=\{x\}$), then $K$ is metrizable. 
	
	\Cref{pr:gruenhage_continuum} also enables us to provide a straightforward proof of the fact that $\wone$ in its standard locally compact interval topology in not Gruenhage, something first observed in \cite{ribarska:88}*{Theorem 5} and later in \cite{stegall:91}*{p.~103}.
	
	\begin{proposition}[{\cite{ribarska:88}*{Theorem 5}}]\label{pr:wone}
	 The ordinal $\wone$ is not Gruenhage. 
	\end{proposition}

	\begin{proof}
	 Let $U_n$, $n \in \omega$, be open subsets of $\wone$. Let $I=\set{n \in \omega}{\wone\setminus U_n \text{ is countable}}$ and let $\alpha < \wone$ be an upper bound of $\bigcup_{n \in I} \wone\setminus U_n$. Then it is easy to see that $(\alpha,\wone)\setminus \bigcup_{n \in \omega\setminus I} U_n$ is a closed unbounded subset of $\wone$ whose points cannot be separated by the $U_n$. By \Cref{pr:gruenhage_continuum}, $\wone$ is not Gruenhage.
	\end{proof}

	One could argue that a class of topological spaces is worthy of study if it is both structured enough to permit meaningful analysis, yet also stable under common standard topological operations. We gather the first crop of topological stability properties of Gruenhage spaces in the next result.
	
	\begin{theorem}[{\cite{smith:09}*{Theorem 23 and Proposition 25}}]\label{th:gruenhage_stability} The following statements hold.
	\begin{enumerate}[label={\upshape{(\roman*)}}]
	 \item\label{th:gruenhage_stability_1} Subspaces of Gruenhage spaces are Gruenhage;
	 \item\label{th:gruenhage_stability_2} products of countably many Gruenhage spaces are Gruenhage;
	 \item\label{th:gruenhage_stability_3} perfect images of Gruenhage spaces are Gruenhage.
	\end{enumerate}
	\end{theorem}

	The only statement requiring a non-trivial proof is \ref{th:gruenhage_stability_3}; below we provide a slightly more streamlined argument than the one given in \cite{smith:09}. 
	
	\begin{proof}[Proof of \Cref{th:gruenhage_stability} \ref{th:gruenhage_stability_3}]
	 Let $X$ be Gruenhage and take sets $R_n$ and families $\mathscr{U}_n$, $n \in \omega$, as in \Cref{pr:gruenhage_equiv} \ref{pr:gruenhage_equiv_3}. By including the singleton families $\{\bigcup\mathscr{U}_n\}$, $n \in \omega$, if necessary, we can assume that, given $n \in \omega$, there exists $m \in \omega$ satisfying $R_m = \bigcup\mathscr{U}_n$.
	 
	 Given non-empty finite $G \subseteq \omega$, define the family of open sets
	 \[
	  \mathscr{V}_G = \set{\bigcap_{n \in G} U_n}{U_n \in \mathscr{U}_n \text{ for all } n \in G}.
	 \]
	Let $f:X \to Y$ be perfect and surjective. Given finite $F,G \subseteq \omega$, $G \neq \varnothing$, and $k \in \omega$, define the family
	\[
	\mathscr{W}_{F,G,k} = \set{Y\setminus f\left(X\setminus\left(\bigcup_{n \in F} R_n \cup \bigcup\mathscr{F} \right) \right)}{\mathscr{F} \subseteq \mathscr{V}_G \text{ and }\card \mathscr{F} = k}
	\]
	Because $f$ is perfect, each element of $\mathscr{W}_{F,G,k}$ is open in $Y$. We claim that the $\mathscr{W}_{F,G,k}$ satisfy \Cref{df:gruenhage_defn}.
	 
	To this end, let $y,y' \in Y$ be distinct and define $K=f^{-1}(y) \cup f^{-1}(y')$, which is compact as $f$ is proper. Let
	 \[
	  \mathscr{A} = \set{I \subseteq \omega}{K\setminus\bigcup_{n \in I} R_n \neq \varnothing}.
	 \]
	Observe that $\mathscr{A}$ is compact in the pointwise topology of $\pow{\omega}$, so $\mathscr{A}$ admits an element $M$ that is maximal with respect to inclusion. Fix $L=K\setminus \bigcup_{n \in M} R_n \neq \varnothing$; by the maximality of $M$, $L \subseteq R_n$ whenever $n \in \omega\setminus M$. 
	
	Set $A=L \cap f^{-1}(y)$ and $B=L \cap f^{-1}(y')$. First, we should dispense with the case that $A$ or $B$ is empty. If $A=\varnothing$ then, by compactness, there exists finite $F \subseteq M$ such that $f^{-1}(y) \subseteq \bigcup_{n \in F} R_n$ (and $f^{-1}(y') \not\subseteq \bigcup_{n \in F} R_n$). Given arbitrary non-empty finite $G \subseteq \omega$, we observe that the only element of $\mathscr{W}_{F,G,0}$ is
	\[
	W:=Y\setminus f\left(X\setminus \bigcup_{n \in F} R_n\right),
	\]
	and $\{y,y'\} \cap W = \{y\}$. We reach a similar conclusion if $B = \varnothing$ instead.
	
	Hereafter we assume $A,B \neq \varnothing$. Given distinct $x,x' \in L$, there exist $n\in\omega$ and $U_0 \in \mathscr{U}_n$ such that $\{x,x'\} \cap U_0$ is a singleton. Notice that $n \in M$, else $x,x' \in L \subseteq R_n \subseteq U_0$, which isn't the case. Also observe that $L \subseteq \bigcup\mathscr{U}_n$; indeed, there exists $m \in \omega$ such that $R_m = \bigcup\mathscr{U}_n$, and we must have $m \in \omega\setminus M$ lest $\{x,x'\} \cap U_0 \subseteq \{x,x'\} \cap R_m = \varnothing$. Therefore $\set{L \cap U}{U \in \mathscr{U}_n}$ partitions $L$, meaning that $\{x,x'\} \cap U=\{x\}$ for some unique $U \in \mathscr{U}_n$ (so we obtain $T_1$-separation of distinct points in $L$).
	
	Fix $x' \in B$. By compactness of $A$ and the observations above, there exists a non-empty finite set $G \subseteq \omega$ such that $\set{L \cap U}{U \in \mathscr{U}_n}$ partitions $L$ whenever $n \in G$ and, given $x \in A$, there exist $n \in G$ and (unique) $U \in \mathscr{U}_n$ satisfying $\{x,x'\} \cap U = \{x\}$. Then note that $\set{L \cap V}{V \in \mathscr{V}_G}$ also partitions $L$ and, given $x \in A$, there exists (unique) $V \in \mathscr{V}_G$ satisfying $\{x,x'\} \cap V = \{x\}$. Therefore, by compactness of $A$, there exists a finite subset $\mathscr{F} \subseteq \mathscr{V}_G$, which we choose to have minimal cardinality $k$, satisfying $A \subseteq \bigcup\mathscr{F}$ and $x' \notin \bigcup\mathscr{F}$. Similarly to above, by compactness there exists finite $F \subseteq M$ satisfying
	\[
	 f^{-1}(y) \subseteq \bigcup_{n \in F} R_n \cup \bigcup\mathscr{F} \quad\text{and}\quad x' \notin \bigcup_{n \in F} R_n \cup \bigcup\mathscr{F}.
	\]
	Then
	\[
	 W:=Y\setminus f\left(X\setminus \bigcup_{n \in F} R_n \cup \bigcup\mathscr{F}\right) \in \mathscr{W}_{F,G,k}
	\]
    and $\{y,y'\} \cap W = \{y\}$. Moreover, $W$ is the only element of $\mathscr{W}_{F,G,k}$ that contains $y$. Indeed, let $\mathscr{G} \subseteq \mathscr{V}_G$ have cardinality $k$ and satisfy
    \[
     y \in f\left(X\setminus \bigcup_{n \in F} R_n \cup \bigcup\mathscr{G}\right).
    \]
    Then $A \subseteq \bigcup\mathscr{G}$, and as $\set{L \cap V}{V \in \mathscr{V}_G}$ partitions $L$ we must have $\mathscr{F} \subseteq \mathscr{G}$ by minimality of $k$; whence $\mathscr{G}=\mathscr{F}$ as required.
	\end{proof}

We will present more topological properties of Gruenhage spaces in \Cref{sect:descendents}. We defer the results until then because they apply to wider classes of spaces. Stability properties of Gruenhage compact spaces of a functional-analytic nature will be given in Section \ref{sect:funct_anal}.
	
We finish this section by mentioning some more examples of Gruenhage compact spaces. First, we recall the space $\mathcal{K}$, known to Banach space theorists as Kunen's compact $S$-space \cite{negrepontis:84}*{Section 7}. This is a subset of $[0,1]$ endowed with a complicated topology that requires CH in its construction. This space, and variants thereof, have been useful for producing counterexamples in Banach space theory \cite{fabian:11}*{Chapter 14}.  While it isn't explicitly stated in \cite{negrepontis:84}*{Section 7}, the topology of $\mathcal{K}$ can be chosen to be finer than the topology it inherits from $[0,1]$. Since, trivially, any space with a finer topology than a Gruenhage space is again Gruenhage, we obtain the next example.

\begin{example}[\cite{orihuela:smith:troyanski:12}*{Proposition 4.7}]
 The Kunen space $\mathcal{K}$ (with a finer topology as described above) is Gruenhage.
\end{example}

There is another set of examples of compact spaces constructed using a combinatorial argument (this time assuming the axiom $\mathfrak{b}=\aleph_1$) that are again trivially Gruenhage because the topology refines that of $\R$.

\begin{example}[\cite{orihuela:smith:troyanski:12}*{Proposition 4.8}]
The spaces in \cite{todorcevic:89}*{Theorem 2.5} are Gruenhage.
\end{example}
		
\section{Functional-analytic properties of Gruenhage compact spaces}\label{sect:funct_anal}

In this section we shall confine our attention to Gruenhage compact spaces $K$, for we will be interested in exploring the properties of the associated dual Banach space $C(K)^*$ and, more generally, dual Banach spaces $X^*$ that are generated by a $w^*$-Gruenhage compact subset.

In order to prove the results in the section we need to do some preparatory work, which culminates in Lemma \ref{lm:families}. A similar process is presented in \cite{smith:09}*{pp.~747--748}, but we take a different approach here. Let $R_n \subseteq K$ and families $\mathscr{U}_n$, $n \in \omega$, satisfy \Cref{pr:gruenhage_equiv} \ref{pr:gruenhage_equiv_3}. Define the new families
\[
 \mathscr{P}_n = \textstyle \{R_n,K\setminus\bigcup\mathscr{U}_n\} \cup \set{U\setminus R_n}{U \in \mathscr{U}_n}, \quad n\in\omega.
\]
By \Cref{pr:gruenhage_equiv} \ref{pr:gruenhage_equiv_3}, the elements of each $\mathscr{P}_n$ are pairwise disjoint and have union $K$. The same can be said of the families $\mathscr{D}_n$, where
\[
 \mathscr{D}_n = \set{P_1 \cap \ldots \cap P_n}{P_i \in \mathscr{P}_i,\,1 \leq i \leq n}, \quad n \in \omega.
\]
If $k \leq n$ then $\mathscr{D}_n$ refines $\mathscr{D}_k$. The next lemma follows by a simple appeal to the inner regularity of Radon measures (cf.~\cite{bogachev:07}*{Proposition 7.2.2 (i)}); its proof is given for completeness.

\begin{lemma}[{cf.~\cite{smith:09}*{Lemma 5}}]\label{lm:measure_lemma}
 Let $\mu \in C(K)^*$ be positive and let $n \in \omega$. Then $\pn{\mu}{1}=\sum_{D \in \mathscr{D}_n} \mu(D)$.
\end{lemma}

\begin{proof}
The statement will follow by induction on $n \in \omega$ provided can show that, given positive $\mu\in C(K)^*$ and a Borel set $E \subseteq K$, we have $\mu(E)=\sum_{P \in \mathscr{P}_n} \mu(E \cap P)$. For this, it is sufficient to verify that
 \[
  \mu(F) = \sum_{U \in \mathscr{U}_n} \mu(E \cap U\setminus R_n),
 \]
 where $F:=E\cap(\bigcup\mathscr{U}_n)\setminus R_n $. Let $\ep>0$. By inner regularity, let $M \subset F$ be compact and satisfy $\mu(F\setminus M) < \ep$. By compactness, there exists a finite set $\mathscr{F} \subseteq \mathscr{U}_n$ such that $M \subset \bigcup\set{E \cap U\setminus R_n}{U \in \mathscr{F}}$, whence
 \[
  \mu(F) - \sum_{U \in \mathscr{F}} \mu(E \cap U\setminus R_n) \leq \mu(F\setminus M) < \ep. \qedhere
 \]
\end{proof}

Since our families $\mathscr{D}_n$ are refining and partition $K$, the family $\mathscr{D}:=\bigcup_{n \in \omega} \mathscr{D}_n$ is a $\pi$-system:~$D \cap D' \in \mathscr{D}$ whenever $D,D' \in \mathscr{D}$. Elements of the proof of the next result appear in that of \cite{stegall:91}*{Lemma 7.6}.

\begin{lemma}[{cf.~\cite{smith:09}*{Lemma 6}}]\label{lm:norming}
 Let $\mu \in C(K)^*$ and $\ep>0$. Then there exist a finite set $\mathscr{F} \subseteq \mathscr{D}$ of pairwise disjoint non-empty sets, and signs $s_D \in \{\pm 1\}$, $D \in \mathscr{F}$, such that
 \[
 \pn{\mu}{1} - \sum_{D \in \mathscr{F}}s_D\mu(D) < \ep.
 \]
 In particular, if $\mu \in C(K)^*$ satisfies $\mu(D)=0$ for all $D \in \mathscr{D}$, then $\mu=0$.
\end{lemma}

\begin{proof}
 Let $\mu \in C(K)^*$ and $\ep>0$. By Lemma \ref{lm:measure_lemma} and the inner regularity of $|\mu|$, for each $n \in \omega$, there exist a finite set $\mathscr{L}_n\subseteq \mathscr{D}_n$ and compact $K_D \subseteq D$, $D \in \mathscr{L}_n$, such that 
 \[
  \pn{\mu}{1} - \sum_{D \in \mathscr{L}_n}|\mu|(K_D) < 2^{-n-1}\ep.
 \]
 Define the compact set $M=\bigcap_{n \in \omega}\left(\bigcup_{D \in \mathscr{L}_n} K_D \right)$. Then $|\mu|(K\setminus M) < \ep$. Given $n \in \omega$ and $D \in \mathscr{L}_n$, set $M_D = M \cap K_D$. Then define $\mathscr{G}_n=\set{D \in \mathscr{L}_n}{M_D \neq \varnothing}$, $n \in \omega$, and set $\mathscr{G}=\bigcup_{n\in\omega}\mathscr{G}_n$. The families $\set{M_D}{D \in \mathscr{G}_n}$, $n \in \omega$, form a refining sequence of pairwise disjoint clopen subsets of $M$, each having union $M$. It follows that $\mathscr{M}:=\set{M_D}{D \in \mathscr{G}}$ is a $\pi$-system and, as $\bigcup_{n\in\omega}\mathscr{U}_n$ separates points of $K$, the points of $M$ are separated by $\mathscr{M}$. Therefore, if we set $W=\set{\ind{M_D}}{D \in \mathscr{G}} \subseteq C(M)$, then $\aspan(W)$ is a subalgebra of $C(M)$ that contains the constant functions and separates points of $M$. Consequently, $C(M)=\cl{\aspan}^{\pndot{\infty}}(W)$ by the Stone-Weierstrass theorem.
 
Given \eqref{eq:variation} and the above, there exists $f \in \aspan(W)$ such that $\pn{f}{\infty} \leq 1$ and $|\mu|(M) - \lint{M}{}{f}{\mu} < \ep$. Again, from above, there exist a finite set $\mathscr{F}\subseteq \mathscr{G}$ and $a_D \in \R$, $D \in \mathscr{F}$, such that $f=\sum_{D \in \mathscr{F}} a_D \ind{M_D}$. As the $\mathscr{G}_n$ are refining, we can assume $\mathscr{F} \subseteq \mathscr{G}_n$ for some $n \in \omega$, which we fix hereafter. Since $M_D \neq \varnothing$ for all $D \in \mathscr{F}$, we have $|a_D| \leq \pn{f}{\infty} \leq 1$. Thus, if we set $s_D = \pm 1$ so that $s_D\mu(M_D)=|\mu(M_D)|$, $D \in \mathscr{F}$, then
\[
 |\mu|(M) - \sum_{D \in \mathscr{F}} s_D\mu(M_D) \leq |\mu|(M) - \sum_{D \in \mathscr{F}} a_D\mu(M_D) = |\mu|(M) - \lint{M}{}{f}{\mu} < \ep.
\]
Because the $D \in \mathscr{F}$ are pairwise disjoint, we have $D\setminus M_D = D\setminus M$ for all such $D$, thus
\[
 \sum_{D \in \mathscr{F}}|\mu|(D\setminus M_D) = \sum_{D \in \mathscr{F}}|\mu|(D\setminus M) \leq |\mu|(K\setminus M) < \ep.
\]
Combining these estimates yields
\[
 \pn{\mu}{1}-\sum_{D \in \mathscr{F}} s_D\mu(D) \leq \mu(K\setminus M) + |\mu|(M)- \sum_{D \in \mathscr{F}} s_D\mu(M_D) + \sum_{D \in \mathscr{F}}|\mu|(D\setminus M_D) < 3\ep. \qedhere
\]
\end{proof}

For our final lemma, recall that if two probability measures agree on a $\pi$-system of measurable subsets of some set $X$, then they agree on the $\sigma$-algebra generated by that $\pi$-system \cite{bogachev:07}*{Lemma 1.9.4}. Thus, if $\mu$ is a signed measure that vanishes on a $\pi$-system of measurable sets which contains $X$, then $\mu$ vanishes on the $\sigma$-algebra generated by that $\pi$-system.

\begin{lemma}\label{lm:families}
 Let $K$ be a Gruenhage compact space. Then there exist open sets $R_n$ and families $\mathscr{U}_n$, $n \in \omega$, that satisfy \Cref{pr:gruenhage_equiv} \ref{pr:gruenhage_equiv_3}, and have the property that if $\mu \in C(K)^*$ satisfies $\mu(U)=0$ for all $U \in \bigcup_{n\in\omega}\mathscr{U}_n$, then $\mu=0$.
\end{lemma}

\begin{proof}
 As $K$ is Gruenhage, we can take open sets $R_n$ and families $\mathscr{U}_n$, $n \in \omega$, that satisfy \Cref{pr:gruenhage_equiv} \ref{pr:gruenhage_equiv_3}. By adding new families if necessary, we can assume that, given $n \in \omega$, there exist $j,k,m \in \omega$ such that $\{R_n\}=\mathscr{U}_j=\mathscr{U}_k$, $\{\bigcup\mathscr{U}_n\}=\mathscr{U}_m$, $R_j=R_n$ and $R_k=R_m=\varnothing$. Define $S_\varnothing=\varnothing$, $\mathscr{V}_\varnothing = \{K\}$ and, given non-empty finite $F \subseteq \omega$, set 
 \[
  S_F = \bigcup_{i \in F} R_i \quad\text{and}\quad \mathscr{V}_F = \set{\bigcap_{n \in F}\bigg(U_n \cup \bigcup_{i \in F\setminus\{n\}} R_i \bigg)}{U_n \in \mathscr{U}_n,\, n \in F}.
 \]
 It is easy to check that the sets $S_F$ and $\mathscr{V}_F$, $F \subseteq \omega$ finite, also satisfy \Cref{pr:gruenhage_equiv} \ref{pr:gruenhage_equiv_3}. Also, we can see that, given non-empty finite $F \subseteq \omega$, there exists non-empty finite $G \subseteq \omega$ such that $\mathscr{V}_G=\{S_F\}$. Indeed, pick $n \in F$, let $j \in \omega$ such that $\{R_n\}=\mathscr{U}_j$ and $R_j=R_n$, and set $G=\{j\} \cup F\setminus\{n\}$.
 
Next, we define the following $\pi$-system. Set
\[
 \mathscr{H} = \{\varnothing,K\} \cup \set{\bigcap_{n \in F} U_n\setminus R_n}{F \subseteq \omega \text{ is finite},\, F \neq \varnothing\text{, and }U_n \in \mathscr{U}_n \text{ for all } n \in F}.
\]
To see that $\mathscr{H}$ is a $\pi$-system, we let $H,H' \in \mathscr{H}$. If $H=\varnothing$ or $H=K$ then clearly $H \cap H' \in \mathscr{H}$, so we assume otherwise; likewise for $H'$. Let $F,F' \subseteq \omega$ be finite, non-empty sets, and pick $U_n \in \mathscr{U}_n$, $n \in F$, and $U'_n \in \mathscr{U}_n$, $n \in F'$, such that $H = \bigcap_{n \in F} U_n\setminus R_n$ and $H' = \bigcap_{n \in F'} U'_n\setminus R_n$. If $U_n \neq U'_n$ for some $n \in F \cap F'$, then $U_n \cap U'_n = R_n$, giving $H \cap H' = \varnothing$. If not, then
\[
 H \cap H' \in \bigg(\bigcap_{n \in F} U_n\setminus R_n\bigg) \cap \bigg(\bigcap_{n \in F'\setminus F} U'_n\setminus R_n \bigg) \in \mathscr{H}.
\]

Now suppose that $\mu \in C(K)^*$ and $\mu(V)=0$ for all $V \in \bigcup_F \mathscr{V}_F$, where the $F$ range over the finite subsets of $\omega$. Then $\mu(H)=0$ for all $H \in \mathscr{H}$. Indeed, let $H \in \mathscr{H}$. If $H=K$ then $\mu(H)=0$ because $K \in \mathscr{V}_\varnothing$. Else, let $F \subseteq \omega$ be finite and non-empty, and pick $U_n \in \mathscr{U}_n$, $n \in F$, such that $H = \bigcap_{n \in F} U_n\setminus R_n$. Set
\[
 V = \bigcap_{n \in F}\bigg(U_n \cup \bigcup_{i \in F\setminus\{n\}} R_i \bigg) \in \mathscr{V}_F
\]
and pick finite, non-empty $G \subseteq \omega$ satisfying $\mathscr{V}_G=\{S_F\}$. Then $H=V\setminus S_F$, giving $\mu(H)=0$.

Finally, we let $\mathscr{A}$ be the algebra of sets generated by $\mathscr{H}$. Evidently $\mathscr{P}_n \subseteq \mathscr{A}$ for all $n \in \omega$, and thus $\mathscr{D} \subseteq \mathscr{A}$. Then $\mu(D)=0$ for all $D \in \mathscr{D}$ because $\mathscr{H}$ is a $\pi$-system containing $K$ and $\mu(H)=0$ for all $H \in \mathscr{H}$. Therefore $\mu=0$ by \Cref{lm:norming}.
\end{proof}

We are ready to start proving more properties enjoyed by Gruenhage compact spaces. Given a class of compact spaces $K$, it is of interest to know whether the dual unit ball $(B_{C(K)^*},w^*)$ of signed Radon measures of total variation at most one also belongs to the class. This is true in the case of Eberlein, Gul'ko  and descriptive compact spaces -- see \cite{amir:lindenstrauss:68}*{Theorem 2}, \cite{talagrand:79}*{Th\'eor\`eme 3.6} and \cite{raja:03}*{Theorem 1.3}, respectively.

Recall that, given a compact space $K$ and an open set $U \subseteq K$, the function $\mu \mapsto \mu^+(U)$, $\mu \in C(K)^*$, is $w^*$-lower semicontinuous because
\begin{equation}\label{eq:lsc}
 \mu^+(U) = \sup\set{\lint{K}{}{f}{\mu}}{f \in C(K),\;0 \leq f \leq 1 \text{ and $f$ vanishes on $K\setminus U$}}.
\end{equation}
Similarly, $\mu \mapsto \mu^-(U)$ is $w^*$-lower semicontinuous.

\begin{proposition}[{\cite{smith:09}*{Proposition 24}}]\label{pr:gruenhage_ball}
If $K$ is Gruenhage compact then so is $B_{C(K)^*}$.
\end{proposition}

\begin{proof}
Let $R_n$ and $\mathscr{U}_n$, $n\in\omega$, be the sets and families furnished by \Cref{lm:families}. Given $U \in \mathscr{U}_n$ and $q \in (0,1) \cap \Q$, by \eqref{eq:lsc} the set
\[
 V^+_{U,n,q} := \set{\mu \in B_{C(K)^*}}{\mu^+(U) > q}                                                                                                    
\]
is $w^*$-open in $B_{C(K)^*}$. Define $\mathscr{V}^+_{n,q}=\{V^+_{U,n,q} \,:\, U \in \mathscr{U}_n\}$. Analogously, define $V^-_{U,n,q}$, $U \in \mathscr{U}_n$ and $\mathscr{V}^-_{n,q}$. We claim that the families $\mathscr{V}^\pm_{n,q}$ satisfy \Cref{df:gruenhage_defn}, meaning that $B_{C(K)^*}$ is Gruenhage. 

Let $\mu,\nu \in B_{C(K)^*}$ be distinct. Either $\mu^+ \neq \nu^+$ or $\mu^- \neq \nu^-$. We suppose that the former holds; if the latter holds then we repeat the argument below using the families $\mathscr{V}^-_{n,q}$. By Lemma \ref{lm:families}, there exist $n \in \omega$ and $U_0 \in \mathscr{U}_n$ such that $\mu^+(U_0) \neq \nu^+(U_0)$. Without loss of generality, let $q \in (0,1) \cap \Q$ such that $\mu^+(U_0) < q < \nu^+(U_0)$. Then $\{\mu,\nu\} \cap V^+_{U_0,n,q} = \{\nu\}$. Now assume $\mu \in V^+_{U,n,q}$ for some $U \in \mathscr{U}_n$. Then $\mu^+(R_n) \leq \mu^+(U_0) < q < \mu^+(U)$, which implies
\[
 \mu^+(U\setminus R_n)=\mu^+(U)-\mu^+(R_n) > q-\mu^+(R_n)>0.
\]
Since the sets $U\setminus R_n$, $U \in \mathscr{U}_n$, are pairwise disjoint, there can only be finitely many $U \in \mathscr{U}_n$ for which $\mu \in V^+_{U,n,q}$.
\end{proof}

Finally, we can present the main renorming result from \cite{smith:09}. Given the above we can provide a proof that relies on less specialized techniques than the original. We call a (semi)norm $\ndot$ on $C(K)^*$ a \emph{lattice (semi)norm} if it respects the natural lattice structure on $C(K)^*$, i.e.~$\n{\mu} \leq \n{\nu}$ whenever $|\mu| \leq |\nu|$.

\begin{theorem}[{\cite{smith:09}*{Theorem 7}}]\label{th:Gruenhage_strictly_convex}
 If $K$ is Gruenhage compact then $C(K)^*$ admits an equivalent dual lattice norm.
\end{theorem}

\begin{proof}
 Let $R_n$ and $\mathscr{U}_n$, $n\in\omega$, be the sets and families given by \Cref{lm:families}. Given $m,n \in \omega$, we define the seminorm $\ptndot{m,n}$ on $C(K)^*$ by
 \[
  \pn{\mu}{m,n} = \sup\set{|\mu|\bigg(R_n \cup \bigcup\mathscr{F} \bigg)}{\mathscr{F} \subseteq \mathscr{U}_n \text{ and }\card{\mathscr{F}} \leq m}.
 \]
 It is evident that $\pndot{m,n} \leq \pndot{1}$ for all $m,n \in \omega$. By \eqref{eq:lsc}, these seminorms are $w^*$-lower semicontinuous. Moreover, it is clear that $|\mu| \leq |\nu|$ implies $\pn{\mu}{m,n} \leq \pn{\nu}{m,n}$. Now define $\ndot$ on $C(K)^*$ by
 \[
  \n{\mu}^2 = \pn{\mu}{1}^2 + \sum_{m,n \in \omega} 2^{-m-n} \pn{\mu}{m,n}^2.
 \]
 This function is a lattice norm on $C(K)^*$ satisfying $\pndot{1} \leq \ndot \leq \sqrt{5}\pndot{1}$, so it is equivalent to the canonical variation norm on $C(K)^*$. Given that it is a sum of $w^*$-lower semicontinuous functions, it is in addition a dual norm by \Cref{pr:lsc}.
 
 It remains to show that $\ndot$ is strictly convex. Let $\mu,\nu \in C(K)^*$ satisfy $\n{\mu}=\n{\nu}=\frac{1}{2}\n{\mu+\nu}$. We show that $\mu=\nu$. The first step is to show that $|\mu|=|\nu|$. For a contradiction, let us suppose otherwise. By \Cref{lm:parallelogram}, we have $\pn{\mu}{m,n}=\pn{\nu}{m,n}=\frac{1}{2}\pn{\mu+\nu}{m,n}$ for all $m,n \in \omega$. In particular, $|\mu|(R_n)=\pn{\mu}{0,n}=\pn{\nu}{0,n}=|\nu|(R_n)$ for all $n\in\omega$. As $|\mu| \neq |\nu|$, by \Cref{lm:families} there exist $n \in \omega$ and $U_0 \in \mathscr{U}_n$ such that $|\mu|(U_0) \neq |\nu|(U_0)$. Given $U \in \mathscr{U}_n$, let $a_U=\max\{|\mu|(U),|\nu|(U)\}$. Then $r:=a_{U_0}-|\mu|(R_n)=a_{U_0}-|\nu|(R_n)>0$. If we set
 \[
 \mathscr{H} = \set{U \in \mathscr{U}_n}{a_U \geq a_{U_0}},
 \]
 then we see that $\mathscr{H}$ is finite. Indeed, let $U \in \mathscr{H}$. If $|\mu|(U) \geq |\nu|(U)$ then
 \[
  |\mu|(U\setminus R_n) = |\mu|(U)-|\mu|(R_n) \geq a_{U_0} - |\mu|(R_n) = r, 
 \]
 and likewise $|\nu|(U\setminus R_n) \geq r$ if $|\nu|(U) \geq |\mu|(U)$. Since the sets $U\setminus R_n$, $U \in \mathscr{U}_n$, are pairwise disjoint, it follows that $\mathscr{H}$ is finite.

Now let $a=\max\set{a_U}{U \in \mathscr{H} \text{ and }|\mu|(U) \neq |\nu|(U)} \geq a_{U_0}$. We define the finite (possibly empty) set
 \[
  \mathscr{F}_0 = \set{U \in \mathscr{H}}{|\mu|(U) = |\nu|(U) \geq a},
 \]
and in addition
\[
 \mathscr{F}_1 = \set{U \in \mathscr{H}}{|\mu|(U) < |\nu|(U) = a} \quad\text{and}\quad \mathscr{F}_2 = \set{U \in \mathscr{H}}{|\nu|(U) < |\mu|(U) = a}.
\]
Either $\mathscr{F}_1$ or $\mathscr{F}_2$ is non-empty; without loss of generality we assume $\mathscr{F}_1 \neq \varnothing$.
 
Define $\mathscr{F}=\mathscr{F}_0 \cup \mathscr{F}_1$ and set $m=\card\mathscr{F}$. We claim that the supremum in the definition of $\pn{\mu+\nu}{m,n}$ is attained, and attained only by $\mathscr{F}$. Indeed, let $\mathscr{G}\subseteq\mathscr{U}_n$ have cardinality at most $m$, with $\mathscr{G}\neq\mathscr{F}$. Then $\mathscr{F}\setminus\mathscr{G} \neq \varnothing$. Given $U \in \mathscr{F}$ and $V \in \mathscr{U}_n\setminus\mathscr{F}$, we estimate
\[
 |\nu|(U\setminus R_n) = |\nu|(U)-|\nu|(R_n) \geq a-|\nu|(R_n) > |\nu|(V)-|\nu|(R_n) = |\nu|(V\setminus R_n),
\]
and as $\card\mathscr{F}\setminus\mathscr{G} \geq \max\{1,\card\mathscr{G}\setminus\mathscr{F}\}$, we further estimate
\[
 |\nu|\bigg(R_n \cup \bigcup\mathscr{F} \bigg) - |\nu|\bigg(R_n \cup \bigcup\mathscr{G} \bigg) = \sum_{U \in \mathscr{F}\setminus\mathscr{G}} |\nu|(U\setminus R_n) - \sum_{V \in \mathscr{G}\setminus\mathscr{F}} |\nu|(V\setminus R_n) > 0.
\]
We conclude that
\[
 |\mu+\nu|\bigg(R_n \cup \bigcup\mathscr{G} \bigg) \leq \pn{\mu}{m,n} + |\nu|\bigg(R_n \cup \bigcup\mathscr{G} \bigg) < \pn{\mu}{m,n} + \pn{\nu}{m,n} = \pn{\mu+\nu}{m,n}.
\]
This completes the proof of the claim.

However, we can also estimate
\begin{align*}
 |\mu|\bigg(R_n \cup \bigcup\mathscr{F} \bigg) &= |\mu|(R_n) + \sum_{U \in \mathscr{F}_0} |\mu|(U\setminus R_n) + \sum_{U \in \mathscr{F}_1} |\mu|(U\setminus R_n)\\ 
 &< |\nu|(R_n) + \sum_{U \in \mathscr{F}_0} |\nu|(U\setminus R_n) + \sum_{U \in \mathscr{F}_1} |\nu|(U\setminus R_n)\\
 &= |\nu|\bigg(R_n \cup \bigcup\mathscr{F} \bigg) \leq \pn{\nu}{m,n},
\end{align*}
which yields the contradiction
\[
 \pn{\mu+\nu}{m,n} = |\mu+\nu|\bigg(R_n \cup \bigcup\mathscr{F} \bigg) \leq |\mu|\bigg(R_n \cup \bigcup\mathscr{F} \bigg) + \pn{\nu}{m,n} < 2\pn{\nu}{m,n}.
\]
Therefore $|\mu|=|\nu|$.

Set $\lambda=\frac{1}{2}(\mu+\nu)$. Then $\n{\mu}=\n{\lambda}=\frac{1}{2}\n{\mu+\lambda}$, so we can repeat the argument above to obtain $\frac{1}{2}|\mu+\nu|=|\lambda|=|\mu|=|\nu|$. This implies $\mu=\nu$ via a lattice argument. Set $\tau=\mu^+-\nu^-$. Then $|\mu|=|\nu|$  implies $\tau=\nu^+-\mu^-$, meaning $2\lambda=\mu+\nu=2\tau$. Hence $\mu^++\mu^-=|\mu|=|\tau|=\tau^++\tau^-$. In addition we see that $\tau^+ = (\mu^+-\nu^-)^+ \leq \mu^+$ and $\tau^- = (-\tau)^+ = (\mu^--\nu^+)^+ \leq \mu^-$. Therefore $\mu^+=\tau^+$ and $\mu^-=\tau^-$, giving $\mu=\tau=\lambda$ and hence $\mu=\nu$. This completes the proof.
\end{proof}

Theorem \ref{th:Gruenhage_strictly_convex} can be generalized to what we can call \emph{Gruenhage generated} dual spaces.

\begin{corollary}[{cf.~\cite{smith:09}*{Corollary 10}}]\label{co:gruenhage_generated}
Let the dual Banach space $X^*$ satisfy $X^*=\cl{\aspan}^{\ndot}(K)$, where $(K,w^*)$ is Gruenhage compact. Then $X^*$ admits an equivalent strictly convex dual norm.
\end{corollary}

\begin{proof}
As $K$ is $w^*$-compact, it is bounded by the uniform boundedness theorem. Define the bounded linear operator $T:X \to C(K)$ by $(Tx)(f)=f(x)$, $f \in K$, $x \in X$. Then $T$ is injective because $X^*=\cl{\aspan}^{\ndot}(K)$. Indeed, if $x,y \in X$ are distinct then $g(x) \neq g(y)$ for some $g \in X^*$, whence $f(x)\neq f(y)$ for some $f \in K$. It follows that the dual operator $T^*:C(K)^*\to X^*$ has dense range. By Theorem \ref{th:Gruenhage_strictly_convex}, $C(K)^*$ admits an equivalent strictly convex dual norm. Hence by \Cref{th:transfer}, $X^*$ also admits such a norm.
\end{proof}

A few years before \cite{smith:09} was published, the following property of Banach spaces was introduced. 

\begin{definition}[\cite{fabian:montesinos:zizler:04}*{Definition 4}]\label{df:G}
 A Banach space $X$ has property G if there exists a bounded set $\Gamma=\bigcup_{n \in \omega} \Gamma_n \subseteq X$, with the property that whenever $f,g \in B_{X^*}$ are distinct, there exist $n \in \omega$ and $x \in \Gamma_n$ such that $f(x)\neq g(x)$ and at least one of the sets
\[ 
 \set{y \in \Gamma_n}{|f(y)| > \tfrac{1}{4}|(f-g)(y)|} \quad\text{and}\quad \set{y \in \Gamma_n}{|g(y)| > \tfrac{1}{4}|(f-g)(y)|}
\]
is finite.
\end{definition}

This definition is clearly inspired by \Cref{df:gruenhage_defn}. The authors remark on  \cite{fabian:montesinos:zizler:04}*{p.~454} that property G is related to Gruenhage compact spaces, and go on to prove that if $X$ possesses property G then $X^*$ admits an equivalent strictly convex dual norm \cite{fabian:montesinos:zizler:04}*{Theorem 5}. This result is a corollary of \Cref{co:gruenhage_generated} (see \cite{smith:09}*{Corollary 13}). All one needs to observe is that if $X$ has property G then $(B_{X^*},w^*)$ is Gruenhage compact \cite{smith:09}*{Proposition 12}. We note that each $w^*$-open subset of $B_{X^*}$ that makes up the families constructed in the proof of \cite{smith:09}*{Proposition 12} is in fact a $w^*$-open \emph{slice} of $B_{X^*}$, i.e.~it is the intersection of $B_{X^*}$ with a $w^*$-open half space of the form $\set{f \in X^*}{f(x) > \alpha}$, for some $x \in X$ and $\alpha \in \R$. There is, in principle, a big difference between families consisting solely of such slices, versus general $w^*$-open subsets; this is a point that we return to in \Cref{sect:descendents}.

We finish the section with two more stability properties of Gruenhage compact spaces. The second follows directly from \Cref{th:gruenhage_stability} but, given its functional-analytic flavor, we postponed it until now.

\begin{corollary}[{cf.~\cite{smith:09}*{Proposition 25 (2) and (4)}}]\label{co:funct_anal} The following statements hold.
\begin{enumerate}[label={\upshape{(\roman*)}}]
\item\label{co:funct_anal_1} Let $M,K$ be compact, with $K$ Gruenhage, and let $T:C(M)\to C(K)$ be an isomorphic embedding. Then $M$ is also Gruenhage compact.
\item Let $X^*$ be a dual Banach space and let $K \subseteq X^*$ be Gruenhage compact in the $w^*$-topology. Then the $w^*$-closed symmetric convex hull $\cl{\sconv}^{w^*}(K)$ is also Gruenhage compact.
\end{enumerate}
\end{corollary}

\begin{proof}
 To prove (i), we consider the dual operator $T^*:C(K)^*\to C(M)^*$. Because $T$ is an isomorphic embedding, $T^*$ is an open map, so there exists $r>0$ such that $rB_{C(M)^*}\subseteq T^*B_{C(K)^*}$. By \Cref{pr:gruenhage_ball}, $B_{C(K)^*}$ is Gruenhage compact in the $w^*$-topology. As dual maps are $w^*$-$w^*$-continuous, $B_{C(M)^*}$ is therefore Gruenhage compact by \Cref{th:gruenhage_stability} \ref{th:gruenhage_stability_1} and \ref{th:gruenhage_stability_3}. Given that the map $x \mapsto \delta_x$ that sends $x \in M$ to its corresponding Dirac measure in $B_{C(M)^*}$ is a homeomorphic embedding (with respect to the $w^*$-topology), again by \Cref{th:gruenhage_stability} \ref{th:gruenhage_stability_1} we conclude that $M$ is Gruenhage.
 
 To prove (ii), recall the dual operator $T^*:C(K)^*\to X^*$ from the proof of \Cref{co:gruenhage_generated} (here, it won't necessarily have dense range). Given $f \in K$, it is easy to see that $T^*\delta_f = f$, where $\delta_f$ is the corresponding Dirac measure in $C(K)^*$. Given that $T^*B_{C(K)^*}$ is symmetric, convex and $w^*$-closed, we obtain $\cl{\sconv}^{w^*}(K) \subseteq T^*B_{C(K)^*}$ (moreover, equality holds via the Krein-Milman theorem). The rest of the proof follows similarly to above.
\end{proof}

In this section we focused on Gruenhage compact spaces and strictly convex dual renorming of $C(K)^*$ spaces. Gruenhage compact spaces also play a role in the  construction of equivalent strictly convex dual norms on $X^*$, where $X$ has a (generally uncountable) unconditional basis \cite{smith:troyanski:09}.

\section{Generalizations and descendents of Gruenhage spaces}\label{sect:descendents}

A couple of years before the appearance of \cite{gruenhage:87}, the concept of fragmentability, modelled on the idea of \emph{dentability} in Banach spaces, was introduced by Jayne and Rogers.

\begin{definition}[\cite{jayne:rogers:85}*{Definition 1}]\label{df:fragmentable}
 Let $X$ be a topological space and let $d$ be a metric on $X$. We say that $X$ is \emph{fragmented by $d$} if, given non-empty $E \subseteq X$ and $\ep>0$, there exists an open subset $U \subseteq X$ with the property that $E \cap U \neq \varnothing$ and the diameter $d$-$\diam{E \cap U}<\ep$. If $X$ is fragmented by some metric then we say that $X$ is \emph{fragmentable}.
\end{definition}

This concept has had an enduring impact on the theory of non-separable Banach spaces. We point out that the metric $d$ need not have any special relationship with the topology on $X$, though in practice there is often a natural choice of metric that is lower semicontinuous with respect to the topology, or generates a finer topology, or both. For example, if $X$ is a scattered topological space then it is fragmented by the discrete metric. If $X$ is a subset of a Banach space $Y$ equipped with the $w$-topology (or a dual space $Y^*$ equipped with the $w^*$-topology), and $d$ is the metric induced by the (dual) norm, then $d$ is both lower semicontinuous and finer than the given topology (though $X$ may or may not be fragmented by $d$).

The following is a very useful characterization of fragmentability expressed in terms of increasing well-ordered families of open sets, which dispenses with the need to involve a metric.

\begin{theorem}[\cite{ribarska:87}*{Theorem 1.9}]\label{th:fragmentability}
 A topological space is fragmentable if and only if there exists, for each $n \in \omega$, an increasing transfinite sequence $(U_{\xi,n})_{\xi \in \lambda_n}$ of open subsets of $X$, such that $\set{U_{\xi,n}}{n \in \omega,\, \xi \in \lambda_n}$ separates points. 
\end{theorem}

We present four applications of this characterization.

\begin{theorem}[\cite{ribarska:87}*{Corollaries 1.11 and 2.27} and \cite{fabian:97}*{Theorem 5.1.12 (v)}]\label{th:fragmentable}
 Let $K$ be a compact fragmentable space. Then
 \begin{enumerate}[label={\upshape{(\roman*)}}]
  \item\label{th:fragmentability_complete} there exists a complete metric $d$ on $K$ that fragments $K$ and is finer then the original topology;
  \item\label{th:fragmentability_dense_Gdelta} $K$ admits a dense $G_\delta$ completely metrizable subspace;
  \item $K$ is sequentially compact.
 \end{enumerate}
\end{theorem}

\begin{theorem}[\cite{ribarska:87}*{Proposition 2.2}]\label{th:Gruenhage_fragmentable}
Every Gruenhage space is fragmentable.
\end{theorem}

Thus we obtain Theorem \ref{th:gruenhage} as an immediate corollary. Given \Cref{pr:wone} and the fact that every scattered compact space is fragmentable, we get an indication of the degree to which \Cref{th:fragmentable} \ref{th:fragmentability_dense_Gdelta} generalizes \Cref{th:gruenhage}.

We will provide a proof of \Cref{th:Gruenhage_fragmentable}, though we will do so indirectly, via an intermediate class of spaces that was identified a little more recently. For some time, this author believed that the converse of \Cref{th:Gruenhage_strictly_convex} (minus the property of being a lattice norm) was true, which would have meant that the class of compact spaces $K$ for which $C(K)^*$ admits an equivalent strictly convex dual norm coincided with the class of Gruenhage compact spaces. This turned out to be false, as we shall see. 

\begin{definition}[{\cite{orihuela:smith:troyanski:12}*{Definition 2.6}}]\label{df:star} A topological space $X$ is said to have property {\st} if there exist families $\mathscr{U}_n$, $n \in \omega$, of open subsets of $X$, such that given $x,y \in X$, there exists $n \in \omega$ satisfying
\begin{enumerate}[label={\upshape{(\roman*)}}]
\item\label{df:star_1} $\{x,y\} \cap \bigcup\mathscr{U}_n$ is non-empty;
\item\label{df:star_2} $\{x,y\} \cap U$ is at most a singleton for all $U \in \mathscr{U}_n$.
\end{enumerate}
\end{definition}

This definition is motivated by both Gruenhage's property and the following straightforward result.

\begin{proposition}[{cf.~\cite{orihuela:smith:troyanski:12}*{Theorem 2.7 (i) $\Rightarrow$ (ii)}}]\label{pr:strictly_convex_star}
 Let the dual Banach space $X^*$ admit an equivalent dual strictly convex norm $\tndot$. Then $(X^*,w^*)$ has {\st}.
\end{proposition}

\begin{proof}
Let $\tndot$ denote the predual norm on $X$ also. Let $S=S_{(X,\tndot)}$ denote the unit sphere of $X$ with respect to $\tndot$. Given $x \in S$ and $q \in \Q \cap (0,\infty)$, define the $w^*$-open set
\[
 U_{x,q} = \set{f \in X^*}{f(x) > q},
\]
and then let $\mathscr{U}_q = \set{U_{x,q}}{x \in S}$, $q \in \Q \cap (0,\infty)$. We claim that these families satisfy \Cref{df:star}. Let $f,g \in X^*$ be distinct. Suppose first that $\tn{f}=\tn{g}$. By strict convexity, there exists $q\in\Q$ satisfying $\tn{\frac{1}{2}(f+g)} < q < \tn{f}$. Then $f,g \in \bigcup\mathscr{U}_q$ by the definition of dual norm. However, it is impossible to have $f,g \in U_{x,q}$ for some $x \in S$, because that would yield $q < \frac{1}{2}(f(x)+g(x)) \leq \tn{\frac{1}{2}(f+g)}$. Now suppose $\tn{f} \neq \tn{g}$. Then, without loss of generality, assume $\tn{f} < \tn{g}$. If we pick $q\in\Q$ strictly between $\tn{f}$ and $\tn{g}$, then we observe that $\{f,g\} \cap \bigcup\mathscr{U}_q = \{g\}$, again by the definition of dual norm. 
\end{proof}

We remark on a very special property of the sets $U_{x,q}$ that make up the families $\mathscr{U}_q$ in the proof above. These are not only $w^*$-open sets but moreover $w^*$-open \emph{half spaces} of $X^*$. Given a subset $C$ of a dual Banach space $X^*$, we shall call a subset of $C$ a $w^*$-open \emph{slice} if it is the intersection of $C$ with a $w^*$-open half space of $X^*$. The set $C$ is said to have \emph{{\st} with slices} if there exist families $\mathscr{U}_n$, $n \in \omega$, that satisfy \Cref{df:star} and moreover consist entirely of $w^*$-open slices of $C$. It turns out it is much easier to build strictly convex norms if one has families that consist entirely of slices. Essentially, this is because, if $C$ is convex and $H \subseteq C$ is a slice, then $C\setminus H$ is also convex. See e.g.~\cite{orihuela:smith:troyanski:12}*{Sections 2 and 3} for more details about this.

Evidently, the notion of {\st} with slices requires some additional information about the underlying Banach space. We will return to this idea a little later on, but for now we shall examine {\st} as a purely topological property. First, it should be clear that \Cref{th:gruenhage_stability} \ref{th:gruenhage_stability_1} and \ref{th:gruenhage_stability_2} apply equally to spaces having {\st} (though whether the analogue of \Cref{th:gruenhage_stability} \ref{th:gruenhage_stability_3} also holds is an open problem -- see below). Next, we will use {\st} to provide an indirect proof of \Cref{th:Gruenhage_fragmentable} and find out some more things about the property along the way.

\begin{proposition}[\cite{orihuela:smith:troyanski:12}*{Proposition 4.1}]\label{pr:Gruenhage_implies_star}
If $X$ is Gruenhage then it has {\st}.
\end{proposition}

\begin{proof}
Given a Gruenhage space $X$, take open sets $R_n$ and families $\mathscr{U}_n$, $n \in \omega$, satisfying \Cref{pr:gruenhage_equiv} \ref{pr:gruenhage_equiv_3}. Assume further, by adding new families if necessary, that given $n \in \omega$, there exists $m \in \omega$ such that $\mathscr{U}_m=\{R_n\}$. We show that the families $\mathscr{U}_n$, $n\in\omega$, satisfy \Cref{df:star}. Given distinct $x,y \in X$, there exist $n \in \omega$ and $U \in \mathscr{U}_n$, such that $\{x,y\} \cap U$ is a singleton. Let $m \in \omega$ such that $\mathscr{U}_m=\{R_n\}$. If $x \in R_n$ then $y \notin R_n$ (else $x,y \in R_n \subseteq U$), thus $\{x,y\} \cap V = \{x\}$ for every $V \in \{R_n\} = \mathscr{U}_m$; likewise if $y \in R_n$. So we assume now that $x,y \notin R_n$. Then it is true that $\{x,y\} \cap V$ is at most a singleton for every $V \in \mathscr{U}_n$, since if $y \in V$ then $V \neq U$, and $x \in V$ would imply $x \in U \cap V = R_n$.
\end{proof}

Recall that a space $X$ is said to have a \emph{$G_\delta$-diagonal} if its diagonal $\set{(x,x)}{x \in X}$ is a $G_\delta$ set in $X^2$. It is straightforward to show that $X$ has a $G_\delta$-diagonal if and only if there are families $\mathscr{U}_n$, $n \in \omega$, of open covers of $X$, such that given $x,y \in X$, there exists $n \in \omega$ with the property that $\{x,y\} \cap U$ is at most a singleton for all $U \in \mathscr{U}_n$ \cite{gruenhage:84}*{Theorem 2.2}. 

Thus it follows that every space having a $G_\delta$-diagonal also has {\st}. Moreover, if $X$ has a $G_\delta$-diagonal and $X \cup \{\infty\}$ is a topological space with respect to which $X$ is open (for instance, the 1-point compactification of a locally compact space), then $X \cup \{\infty\}$ has {\st}. According to a result of \v Sne\v\i der, compact spaces having $G_\delta$-diagonals are metrizable \cite{gruenhage:84}*{Theorem 2.13}, so by \Cref{pr:Gruenhage_implies_star} it is evident that {\st} is a strict generalization of the property of having a $G_\delta$-diagonal.

\begin{proposition}[\cite{orihuela:smith:troyanski:12}*{Proposition 4.2}]\label{pr:star_implies_fragmentable}
If $X$ has {\st} then it is fragmentable.
\end{proposition}

\begin{proof}
Let $X$ admit families $\mathscr{U}_n$, $n \in \omega$, that satisfy Definition \ref{df:star}. We well order each $\mathscr{U}_n$ as $(U^n_\xi)_{\xi \in \lambda_n}$. Now define $V^n_\alpha = \bigcup_{\xi \in \alpha+1} U^n_\xi$, for $\alpha \in \lambda_n$. We claim that, given distinct $x,y \in X$, there exist $n \in \omega$ and $\alpha \in \lambda_n$ such that $\{x,y\} \cap V^n_\alpha$ is a singleton. This shows that $X$ is fragmentable by Theorem \ref{th:fragmentability}. Indeed, take $n \in \omega$ satisfying \Cref{df:star} \ref{df:star_1} and \ref{df:star_2}. Then pick the least $\alpha \in \lambda_n$ such that $\{x,y\} \cap U^n_\alpha$ is a singleton. Thus $\{x,y\} \cap U^n_\xi$ must be empty for all $\xi \in
\alpha$, meaning that $\{x,y\} \cap V^n_\alpha = \{x,y\} \cap U^n_\alpha$ is a singleton.
\end{proof}

As Gruenhage spaces have {\st}, they inherit the properties spelled out in the next result. \Cref{th:star_properties} \ref{th:star_properties_1} generalizes a result of Chaber which states that countably compact spaces having $G_\delta$-diagonals are compact; \Cref{th:star_properties} \ref{th:star_properties_2} generalizes \cite{oncina:raja:04}*{Corollary 4.3}.

\begin{theorem}[\cite{orihuela:smith:troyanski:12}*{Theorem 4.3 and Corollary 4.4}]\label{th:star_properties}
The following statements hold.
\begin{enumerate}[label={\upshape{(\roman*)}}]
\item\label{th:star_properties_1} If $X$ is countably compact and has {\st}, then $X$ is compact.
\item\label{th:star_properties_2} If $L$ is locally compact and has {\st}, then $L \cup \{\infty\}$ is countably tight and sequentially closed subsets of $L \cup \{\infty\}$ are closed.
\end{enumerate}
\end{theorem}

\Cref{th:star_properties} \ref{th:star_properties_1} yields a quick proof of the fact that $\wone$ fails to have {\st}, given that this space is countably compact but not compact. This can also be seen in a more elementary way by using the pressing down lemma in an elaboration of the proof of \Cref{pr:wone} -- see \cite{smith:troyanski:10}*{Example 1}. As a byproduct, we see that the class of spaces having {\st} is strictly included in the class of fragmentable spaces. In addition, combining this with \Cref{pr:strictly_convex_star}, we derive the result of Talagrand stating that the dual Banach space $C(\omega+1)^*$ admits no equivalent strictly convex dual norm.

Examples of scattered, locally compact, non-Gruenhage spaces having $G_\delta$-diagonals (and thus having {\st} also) were constructed in \cite{orihuela:smith:troyanski:12}*{Example 2}, assuming additional axioms. A little later, a ZFC example was found.

\begin{theorem}[\cite{smith:11}*{Theorem 2.4}]\label{th:diagonal_non-Gruenhage}
 There is a scattered, locally compact non-Gruenhage space $L$ having a $G_\delta$-diagonal. Consequently, $L \cup \{\infty\}$ has {\st} and is non-Gruenhage.
\end{theorem}

Thus the spaces having {\st} form a class strictly in between the classes of Gruenhage and fragmentable spaces.

Now we see how {\st} is related to the existence of strictly convex dual norms on $C(K)^*$. We begin with a necessary condition.

\begin{proposition}[\cite{orihuela:smith:troyanski:12}*{Proposition 3.2}]\label{pr:strictly_convex_K_star}
 If $C(K)^*$ admits a strictly convex dual norm then $K$ has {\st}.
\end{proposition}

\begin{proof}
Recall that the map that sends $x \in K$ to the corresponding Dirac measure $\delta_x \in C(K)^*$ is a homeomorphic embedding. Given that {\st} is inherited by subspaces, the conclusion now follows from \Cref{pr:strictly_convex_star}.
\end{proof}

We state the related sufficient condition in \cite{orihuela:smith:troyanski:12} without proof.

\begin{theorem}[\cite{orihuela:smith:troyanski:12}*{Theorem 3.1}]\label{th:scattered_star}
Let $K$ be a scattered compact space having {\st}. Then $C(K)^*$ admits a strictly convex dual (lattice) norm.
\end{theorem}

Together, \Cref{th:diagonal_non-Gruenhage,th:scattered_star} show that the converse to \Cref{th:Gruenhage_strictly_convex} is false. \Cref{th:scattered_star} also happens to yield a very partial positive answer to the question of whether {\st} is preserved by taking perfect images. The next result is a very modest generalization of \cite{orihuela:smith:troyanski:12}*{Proposition 4.5}; compare it with \Cref{co:funct_anal} \ref{co:funct_anal_1}.

\begin{proposition}\label{pr:scattered_star_image}
Let $M,K$ be compact, with $K$ scattered and having {\st}, and let $T:C(M)\to C(K)$ be a bounded linear injection. Then $M$ has {\st}.
\end{proposition}

\begin{proof}
By Theorem \ref{th:scattered_star}, let $\ndot$ be an equivalent strictly convex dual norm on $C(K)^*$. As $T$ is an injection, the dual $T^*:C(K)^*\to C(M)^*$ has dense range. By \Cref{th:transfer}, $C(M)^*$ also admits an equivalent strictly convex dual norm. It follows that $M$ has {\st} by \Cref{pr:strictly_convex_K_star}.
\end{proof}

Therefore, if $M,K$ are compact, with $K$ scattered and having {\st}, and $\pi:K\to M$ is continuous and surjective, then $M$ has {\st} because we can apply \Cref{pr:scattered_star_image} to $T:C(M)\to C(K)$ defined by $Tf= f \circ \pi$.

Now we resume the discussion about {\st} with slices. Let $X^*$ be a dual Banach space. As {\st} is inherited by subspaces, if $(X^*,w^*)$ has {\st} then so do $(B_{X^*},w^*)$ and $(S_{X^*},w^*)$. Moreover, by a straightforward geometric argument (cf.~the proof of \cite{orihuela:smith:troyanski:12}*{Theorem 2.7 (iii) $\Rightarrow$ (ii)}), if $(S_{X^*},w^*)$ has {\st} then so does $(X,w^*)$. Exactly the same equivalences hold when {\st} is replaced by {\st} with slices. Because of this, we shall consider {\st} or {\st} with slices with respect to $(B_{X^*},w^*)$ alone.

According to \cite{orihuela:smith:troyanski:12}*{Theorem 2.7}, if $(B_{X^*},w^*)$ has {\st} with slices, $X^*$ admits an equivalent strictly convex dual norm. Is it the case that, in fact, $(B_{X^*},w^*)$ has {\st} with slices? If so, we would obtain a characterization of strictly convex dual renormability in dual Banach spaces in terms of a purely topological property of $(B_{X^*},w^*)$, giving the sought after analogue of \Cref{th:raja} \ref{th:raja_1}.

Given the usual basis of the $w^*$-topology, $w^*$-open subsets of $B_{X^*}$ can be written as unions of intersections of finitely many $w^*$-open slices of $B_{X^*}$. By the Choquet lemma \cite{fabian:11}*{Lemma 3.69}, every \emph{extreme point} of $B_{X^*}$ has a local base of $w^*$-open slices (with respect to the $w^*$-topology). Given ``enough'' extreme points, it is possible to use this fact to replace $w^*$-open sets with $w^*$-open slices. We say that a Banach space $X$ is an \emph{Asplund space} if every separable subspace of $X$ has a separable dual space (this is not the original definition but one of several equivalent statements). Asplund spaces have been the topic of an enormous volume of work in Banach space theory. It has long been known that, given a compact space $K$, $C(K)$ is an Asplund space if and only if $K$ scattered \cite{fabian:11}*{Theorem 14.25}. In the context of Asplund spaces we can give an answer to the question above.

\begin{theorem}[\cite{smith:20}*{Theorem 1.4}]\label{th:Asplund_star}
Let $X$ be an Asplund space. Then the following are equivalent.
\begin{enumerate}[label={\upshape{(\roman*)}}]
\item $X^*$ admits an equivalent strictly convex dual norm;
\item $(B_{X^*},w^*)$ has {\st} with slices;
\item $(B_{X^*},w^*)$ has {\st}.
\end{enumerate}
\end{theorem}
	
	We conclude this survey with some problems which, to the best of the author's knowledge, are open.
	
	\begin{problem}\label{pb:star_star_slices}
	 Let $X^*$ be a dual Banach space. If $(B_{X^*},w^*)$ has {\st} then does it have {\st} with slices?
	\end{problem}
	
	\begin{problem}\label{pb:st_generated}
	 Let $X^*$ be a \emph{{\st}-generated} dual space, that is, $X^*=\cl{\aspan}^{\ndot}(K)$, where $(K,w^*)$ is compact and has {\st}. Does $X^*$ admits an equivalent strictly convex dual norm?
	\end{problem}
	
	\begin{problem}\label{pb:ck_star}
	 Let $K$ be a compact space having {\st}. Does $C(K)^*$ admit an equivalent strictly convex dual norm?
	\end{problem}

	\begin{problem}\label{pb:star_perfect_image}
	 Let $X$ be a Hausdorff space having {\st} and let $\pi:X\to Y$ be a perfect map onto a Hausdorff space $Y$. Does $Y$ have {\st}? 
	\end{problem}

	We remark that a positive answer to \Cref{pb:ck_star} would yield an analogue of \Cref{th:raja} \ref{th:raja_2}, and a positive answer to all other stated problems besides \Cref{pb:star_perfect_image}, together with a few others. Indeed, \Cref{pb:st_generated} would follow by repeating the argument in the proof of \Cref{co:gruenhage_generated}. (For the same reason, by \Cref{th:scattered_star} we know that \Cref{pb:st_generated} has a positive answer if $(K,w^*)$ happens to be scattered.) In turn, straightaway we would obtain a positive answer to \Cref{pb:star_star_slices}. In addition, we would also obtain the analogue of \Cref{pr:gruenhage_ball} for compact spaces having {\st}. Using the argument in the proof of \Cref{pr:scattered_star_image}, we would also get a positive answer to \Cref{pb:star_perfect_image} in the case that $X$ is compact. Finally, this would yield the analogue of \Cref{co:funct_anal} for compact spaces having {\st}.
	
	\bibliography{document}
	
\end{document}